\documentclass[final]{siamltex}
\usepackage{mathrsfs}
\usepackage{amssymb,amsmath,amsfonts}
\usepackage{graphicx}
\usepackage{epsfig}
\usepackage{algorithmic, algorithm}
\usepackage{color}
\usepackage{showkeys}%[notcite, notref]
\usepackage{multirow}
\usepackage[colorlinks=true,citecolor=blue,linkcolor=blue]{hyperref}

\newtheorem{assumption}[theorem]{Assumption}
\newtheorem{remark}[theorem]{Remark}

\newcommand{\kl}[1]{{{#1}}}
\newcommand{\todo}[1]{{#1}}

\definecolor{darkgreen}{rgb}{0,0.7,0}
\definecolor{darkblue}{rgb}{0,0,0.7}
\definecolor{darkred}{rgb}{0.7,0,0}

\newcommand{\G}{N}
\newcommand{\pl}{p}
\newcommand{\ml}{\pi}
\newcommand{\cP}{\mathcal{P}}
\newcommand{\cK}{\mathcal{K}}

\newcommand{\cC}{\mathcal{C}}

\newcommand{\R}{\mathbb{R}}

\newcommand{\bbR}{\mathbb{R}}

\newcommand{\hc}{\widehat{C}}
\newcommand{\hm}{\widehat{m}}

\newcommand{\bbE}{\mathbb E}

\newcommand{\cG}{\mathcal{G}}

\newcommand{\cO}{\mathcal{O}}

\newcommand{\quark}{{\setbox0\hbox{$x$}\hbox to\wd0{\hss$\cdot$\hss}}}

\newcommand{\E}{\mathbb{E}}

\renewcommand{\arraystretch}{2}

\title{{Deterministic %Methods for Filtering: %, part I: 
Mean-field Ensemble Kalman Filtering}}

\author{Kody J.H. Law\thanks{{\tt SRI-UQ Center, CEMSE, KAUST, Thuwal, KSA.} 
CURRENT ADDRESS: Computer Science and Mathematics Division, Oak Ridge National Laboratory, Oak Ridge, TN, USA, 37831
({\tt lawkj@ornl.gov}) }
\and 
 Hamidou Tembine\thanks{SRI-UQ Center, CEMSE, KAUST, Thuwal, KSA ({\tt hamidou.tembine@kaust.edu.sa})}
\and 
Raul Tempone\thanks{SRI-UQ Center, CEMSE, KAUST, Thuwal, KSA ({\tt raul.tempone@kaust.edu.sa})}
}

\begin{document}

\maketitle

\begin{abstract}{The proof of convergence of the standard ensemble Kalman filter (EnKF) from Legland etal. 2011 \cite{le2011large}
is extended to non-Gaussian state space models. 
A density-based deterministic approximation of the mean-field limit EnKF (DMFEnKF) is proposed, consisting of a PDE solver and a quadrature rule. 
Given a certain minimal order of convergence $\kappa$ between the two, this extends to the deterministic filter approximation, which is therefore asymptotically superior to standard EnKF for dimension $d<2\kappa$. 
The fidelity of approximation of the true distribution is also established using an extension of total variation metric to random measures. 
This is limited by a Gaussian bias term arising from non-linearity/non-Gaussianity of the model, which arises in both deterministic and standard EnKF. Numerical results support and extend the theory. }

\end{abstract}


\pagestyle{myheadings}
\thispagestyle{plain}
\markboth{Deterministic Nonlinear Filters and EnKF}{Deterministic Nonlinear Filters and EnKF}

\section{Introduction}
\label{sec:intro}

The filtering problem, referred to in the geophysical community as data assimilation 
\cite{kal03}, consists of obtaining meaningful information sequentially online about 
a signal evolving in time, %or state-space model, 
given noisy observations of that signal.  
From the Bayesian perspective the solution of the filtering problem is given 
by the posterior distribution of the signal given all the previous observations 
\cite{jaz70, BC09, kaipio2005statistical, Stuart10, ApteJV08}.
{The signal is typically modeled by a Markov process in which the observation at a given time is conditionally independent upon the rest of the observations given the observed state at that time.}
This set-up is then referred to as a hidden Markov model \cite{cappe2005inference}.
In the case of linear Gaussian state-space model, the solution is also Gaussian, and therefore
it can be parametrized by its mean and covariance and is given exactly in closed
form by the Kalman filter \cite{kalman1960new}.
In general, other cases must be treated non-parametrically, for example with computational
algorithms.
The optimal filter is a point estimator given by the expected value of the filtering
distribution \cite{jaz70, BC09}.
This can be estimated consistently using particle filtering algorithms 
\cite{BC09, doucet2000sequential, DdFG01}.
Indeed the particle filter consistently approximates 
any quantity of interest, i.e. any conditional expectation. 
Despite being asymptotically consistent in the large particle limit, 
it is well-known that the accuracy of particle filters is hindered by a 
constant that grows exponentially with dimension \cite{BLB08,rebeschini2013can}.
Furthermore, naive bounds indicate the constant may also grow
with time, however if 
the hidden process has a Dobrushin ergodic coefficient \cite{dobrushin1956central}
then a uniform-in-time estimate can be obtained 
\cite{del2001stability}. 
In the geophysical community, the dimension of the state-space is typically
enormous, and so practitioners have resorted to sub-optimal filters
such as the ensemble Kalman filter (EnKF) \cite{evensen1994sequential} and its incarnations.
It has been shown by \cite{le2011large, kwiatkowski2015convergence, mandel2011convergence} 
that some versions of EnKF for models of a particular class converge to a 
{\it mean-field limit}, which is defined herein as a process in which the 
current state depends only on the previous state and 
\todo{{\it the statistics of the process}.} 
Such filters may perform well in high dimensions for small ensembles, 
but are biased in the sense that the mean-field limiting distribution is {\it not} 
the filtering distribution and so
estimators of quantities of interest do not converge to the correct value
in the large ensemble limit.  
As both of these methods depend on random ensembles of particles, 
{the asymptotic approximation error of expectations with respect to 
their respective limits is given by $\cO(N^{-1/2})$ for an ensemble size $N$.}
Interestingly, it has been observed that the signal-tracking error of EnKF may not 
decrease for ensembles larger than 50-100 for the quasi-geostrophic equations 
in \cite{evensen1994sequential}, indicating the
error may become dominated by the bias arising from the linearity assumption.

{The EnKF is a filtering algorithm which was introduced in \cite{evensen1994sequential} 
and later corrected in \cite{burgers1998analysis}.
Herein the filter is developed from the perspective as a 
Monte-Carlo approximation of the 
minimum mean-square error linear estimator 
(Thm. 1, p. 87 and Thm. 3, p. 92 of \cite{luenberger1968optimization}).
\footnote{From the perspective in which the unknown is considered deterministic, 
the best linear estimator is given by the similar 
Gauss-Markov Theorem (Thm. 1, p. 86 \cite{luenberger1968optimization}).}
For a single step in the mean-field limit (ensemble size $N\rightarrow \infty$), 
EnKF returns a random variable 
whose expected value 
has the minimum mean-square error over all 
estimators of the forecasted signal which are linear in the new observation.}
Furthermore, this random variable has as its covariance the error covariance
between the minimum mean-square error linear estimator and the truth.
It can thus be viewed as a best linear approximation of the target for a single step.
It is of course not the unique random variable with this mean and covariance, 
even among those which are linear in the observation.
For linear Gaussian models however, this estimator happens to be the mean of the posterior 
distribution, which is given by a single step of the Kalman filter \cite{kalman1960new}.
Furthermore, in this case the error covariance of the estimator is 
the covariance of the posterior distribution. 
Hence the corresponding random variable is distributed according to the filtering 
distribution, and this holds for all time.
{The mean of the posterior filtering distribution is the minimum mean-square error estimator 
over all square integrable functions of the observation, and is hence known as the optimal filter 
\cite{jaz70}.}
For nonlinear and/or non-Gaussian models 
linear estimators based on the procedure outlined above
{\it do not} yield the mean of the updated posterior filtering distribution, 
which is in general nonlinear in the observations. 
Therefore, in this case linear estimators correspond to sub-optimal
filters.
The theory above was translated into finite-resolution approximations of
sub-optimal filters 
using ensemble approximations in the EnKF algorithm 
\cite{burgers1998analysis},
and 
using random-variable-based deterministic approximations in the algorithms of  
\cite{pajonk2012deterministic, li2009generalized, ernst2014bayesian} and references therein.
The papers \cite{augustin2008polynomial, branicki2012fundamental}
caution against the use of random-variable expansions for the 
forward propagation of uncertainty.
While the mean-field EnKF yields a non-Gaussian approximation, 
the mean and covariance 
{are identical to those obtained by making a Gaussian approximation of the
forecast distribution and updating that Gaussian with the Gaussian observations,
as done for example in  \cite{solonen2012variational}.}
Such procedure would therefore also 
yield a best linear approximation for a single step as defined above. 
The results herein will illustrate that the EnKF estimator may 
actually perform better than such Gaussian approximation 
in recovering the mean and covariance in the long-run.

Great effort has been invested in approximating the filtering distribution
using particle and ensemble methods 
\cite{doucet2000sequential, vanden2012data, chorin2010implicit, santitissadeekorn2014two, reich2012gaussian, hoteit2011particle, ades2013exploration}, 
while decidedly less attention has been invested in deterministic approximations \cite{eyink2004mean, mandel2009ensemble, salman2008hybrid, el2012bayesian, el2012bayesian2, pajonk2012deterministic, li2009generalized, bao2014hybrid}.
{The idea of numerically approximating the 
evolution of the Fokker-Planck equation and imposing the update by 
multiplication and normalization in the continuous-discrete 
setting has been %is not new and was 
done %already 
in the works
\cite{eyink2004mean, mandel2009ensemble}.  It was used in those works, as it will be here, 
as a benchmark against which to evaluate other algorithms developed.} 
%, in one spatial dimension.}
%
%However, 
\todo{Here referred to as the full Fokker Planck filter (FPF), 
 it} is actually also advocated as a legitimate and competitive method 
 which can be lifted to higher-dimensional
problems using sparse-grid parametrization ideas 
(see, for example, \cite{bungartz2004sparse, bao2014hybrid}).
More sophisticated approximations of the density 
should be capable of handling much
higher dimensions.
Furthermore, it is well-known that very 
high-dimensional models may exhibit nonlinearity/instability/non-Gaussianity only 
on low-dimensional manifolds 
\cite{brett2012accuracy, blomker2012accuracy, CK04, sapsis2013blended, hayden2011discrete, uboldi2005developing, law2012evaluating}.  
Therefore, it is conceivable that the space can be decomposed, and such 
approximation of the true filtering distribution may be used on the unstable space, 
while a simpler filter (e.g. ensemble or extended Kalman or even 3DVAR) may be used on the complement.
In certain cases such approaches may be particularly simple.  For example,
in a particular context of Lagrangian data assimilation it has been proposed to 
use a particle filter on the flow state and the Fokker-Planck equation for the evolution 
of observed passive tracers, 
which are conditionally independent given a particular flow state \cite{salman2008hybrid}.

The aim of this work is two-fold: %on the one hand, 
\todo{first, a deterministic
approach 
\footnote{{What is meant by deterministic here is that no random number generation is required 
in the algorithm, and this is afforded by working only at the level of densities.}} 
to solving the %true filtering problem and 
mean-field EnKF (MFEnKF) via 
accurate numerical approximation of the Fokker-Planck equation is proposed, 
%similarly to the FPF of \cite{eyink2004mean, mandel2009ensemble},
and %on the other hand this 
the analogous approximation of the true filter, FPF \cite{eyink2004mean, mandel2009ensemble}, 
is used as a benchmark against which to evaluate other filters numerically, 
in particular the standard EnKF and deterministic EnKF (DMFEnKF). %deterministic mean-field EnKF.  
\todo{The DMFEnKF is {\it not} proposed here as an alternative to the Full FPF, as the latter
will provably always have lower mean square error (MSE) for sufficiently accurate approximation schemes.  
However, the DMFEnKF {\it may} be more robust to errors, as it replaces the normalization step of 
the Full FPF by a linear change of variables and convolution with a Gaussian density.}
Second,} the convergence rates of the standard EnKF 
%is established following  \cite{le2011large}. %using an induction argument. 
%The numerical approximation of the 
and DMFEnKF to the MFEnKF are derived.  %mean-field EnKF 
It is shown that the latter has a faster asymptotic rate of convergence %to the limit 
%than the standard EnKF 
than the former for $d < 2\kappa$, where $\kappa$ is the minimal rate of convergence of the 
numerical approximation of the Fokker-Planck equation and the quadrature rule for $d=1$.
%Furthermore, this may be greatly improved by using more sophisticated numerical
%techniques.  
In short, the order of convergence of the standard density and quadrature approximations 
extends to the filtering density approximation for finite times.
%, and this may be 
%smaller using the deterministic approach rather than the Monte Carlo approach.
%
It is also proven that the value of this faster convergence is limited by 
the Gaussian bias of the MFEnKF %mean-field limiting EnKF 
when the underlying nonlinearity/non-Gaussianity of the 
forward model becomes significant, for example
from longer integration of a nonlinear SDE between observations.
In this case, the approximation of the true filter may be used as an accurate
and effective algorithm also for $d < 2 \kappa$, 
although this is not emphasized in the present work. 
The theoretical results are complemented by numerical experiments which confirm the theory and 
also illustrate exactly where the Gaussian error is greatest. 
In particular, approximating the forecast %predicting 
distribution by a Gaussian gives a less
accurate approximation than standard mean-field EnKF, while approximating the updated 
distribution by a Gaussian after the full nonlinear update gives a more 
accurate approximation.
The results of this paper complement recent results for the non-divergence, in terms of
tracking the true signal, of continuous-discrete and continuous-continuous 
EnKF for a general class of quadratic-dissipative, noise-free, 
and possibly infinite-dimensional dynamics \cite{KLS13},
which may be viewed as a particular class of nonlinear Gaussian state-space model 
(i.e. the case of degenerate dynamics in which the Fokker-Planck evolution in finite dimensions 
would reduce to Liouville equation for the continuous-discrete case).  
The results also complement recent results on convergence of EnKF to MFEnKF, %mean-field limit 
for both the linear \cite{kwiatkowski2015convergence, mandel2011convergence} and nonlinear \cite{le2011large}
Gaussian state-space model cases, by extending those results to the more general non-Gaussian
case considered here and considering for the first time, to the knowledge of the authors, 
the fidelity with which MFEnKF %that mean-field limit 
approximates the true filtering distribution.

The rest of the paper will be structured as follows.  
In Section \ref{sec:dtf} the filtering problem is
introduced and definitions of various notions of its solution are given and discussed. 
In Section \ref{sec:enkf} the ensemble Kalman filter is defined and related to the 
discussion of Section \ref{sec:dtf}.
In Section \ref{sec:fpf} the Fokker-Planck solution approach is introduced, several
algorithms are defined, and the convergence theorem is presented.
In Section \ref{ssec:theo} the majority of the theoretical results are presented.
In Section \ref{sec:numerics} numerical experiments are done with the algorithms
to confirm and extend the theoretical results.
Finally, Section \ref{sec:conc} gives conclusions and future directions.

\section{Filtering}\label{sec:dtf}

In this section a general filtering problem is set up and the meaning of solution
is defined and discussed.  In particular, the filtering distribution is introduced, which 
has the so-called optimal filter as its mean.  Then suboptimal filters are introduced,
including the set of one-step optimal linear  filters, and some standard sub-optimal
filters for general nonlinear Gaussian state-space models.  The relation between 
the latter is highlighted, and this gives segue to the EnKF which will be the focal
sub-optimal filter of this work.

\subsection{Set-up}\label{ssec:setup}

{Throughout 
for any positive-definite $A \in \bbR^{\ell \times \ell}$, we 
introduce the following notations for weighted Mahalanobis inner-product 
$ \langle \cdot,\cdot \rangle_{A}=\langle A^{-\frac12}\cdot,A^{-\frac12}\cdot \rangle$ 
and the resulting norm $|\cdot|_{A}=|A^{-\frac12}\cdot|.$ }

{Let $\mathcal{K}: \bbR^d \times \sigma(\bbR^d) \to \bbR_+$ 
be a generic Markov kernel,  
where $\sigma(\bbR^d)$ is the sigma algebra of measurable subsets of $\bbR^d$.
This simultaneously gives rise to a linear operator on measures/densities 
and functionals so that for $f:\bbR^d \rightarrow \bbR$, 
$(\cK f)(u) = \int_{\bbR^d} \cK(u,dv)f(v)$, and for $\mu: \sigma(\bbR^d) \rightarrow [0,1]$
with density $\rho:\bbR^d \rightarrow \bbR_+$,  $\mu'$ is a measure with density $\rho'$
such that 
$\mu'(A) = \int_{\bbR^d} \mu(du) \cK(u,A) = \int_{\bbR^d} \rho(u) \cK(u,A) du = 
\int_A \int_{\bbR^d} \rho(u) \cK(u,v) du dv$ for all $A \in \sigma(\bbR^d)$.  
This will be written as \kl{$\mu' = \cK^\top \mu$ or $\rho' = \cK^\top \rho$}.
 This work will only concern distributions which have density 
with respect to Lebesgue measure and, as such, densities will be 
used interchangeably with probability measures. 
In particular, for each $u\in \bbR^d$, the measure $\cK(u,\cdot)$ is simultaneously identified
with its density $\cK(u,v)$.  Consider the Markov chain $u=\{u_j\}_{j\geq0}$
defined by }
\begin{eqnarray}
\label{eq:dtf2}
 u_{j+1} &\sim&\mathcal{K}(u_j, \cdot), \;j = 0,1,2,\dots, \\
u_0&\sim& \rho_0. 
\nonumber
\end{eqnarray}
This Markov Chain returns a sequence of random variables
related by the Markov property, 
i.e. $u_k | u_{j} = u_k | u_j, u_{j-1}, \dots, u_0$ for all
$k>j$.

In many applications, models such as \eqref{eq:dtf2} are
supplemented by observations of the system as it evolves.
{As part of the statistical model under consideration in the present work 
%is one in which 
it is assumed that the data is defined as}%related to the signal as follows}
\begin{equation}\label{eq:dtfo}
y_{j}=Hu_{j}+\eta_{j}, \;j=1,2,\dots, 
\end{equation}
where $H:\R^d\to\R^m$ is linear 
\footnote{{The assumption of linear observation operator is made here only for simplicity and 
is easily extended.  
For example, an auxiliary variable can always be introduced so that the resulting extended
system has a linear observation operator.}}
and $\eta=\{\eta_j\}_{j\geq 1}$ 
is an i.i.d. sequence,
independent of $u_0$ and the noise in $\cK$, 
with $\eta_1\sim\G(0,\Gamma)$.  
{The accumulated data is denoted $Y_k=\{y_j\}_{j=1}^k$. }
The objective of filtering is to determine 
information about the conditional, or {\it filtered} random variable 
$u_j|Y_j$.  Its distribution is referred to as the filtering distribution
and recovering either this distribution, or (ambiguously) even just an estimate of it, 
is referred to as filtering.  We will refer to the former problem as the true 
filtering problem.
Often the dependence on $u_0$ is neglected.  
This is reasonable in the case 
that the true filtering problem
is stable, in the sense of forgetting its initial distribution in the large time limit.  

Notice that models of the form \eqref{eq:dtf2} include 
models of the following form as a special case 
\begin{eqnarray}\label{eq:dtf1}
\label{eq:model}
 u_{j+1}&=&\Psi(u_j)+\xi_j, \;j = 0,1,2,\dots, \\
u_0&\sim& \rho_0, %\G(m_0,C_0),
\nonumber
\end{eqnarray}
where  $\Psi:\R^d\to\R^d$, and $\xi=\{\xi_j\}_{j\geq0}$ 
is an i.i.d. sequence, independent of $u_0$,
with $\xi_0\sim\G(0,\Sigma)$. 
Gaussian state-space 
models of the form \eqref{eq:dtf1} are commonly encountered, in particular in the 
data assimilation community.

\subsection{Filtering distribution}
\label{ssec:true}

The true distribution of $u_j|Y_j$, which is our {\it gold standard}, has a recursive 
structure under the given assumptions.  
{Define the unnormalized joint likelihood density $g(u,y) \propto p(y|u)$, 
for a particular pair $(u,y)$, as 
\begin{eqnarray}
g(u,y) = %\frac{\tilde{g}(u,y)}{\int \tilde{g}(u,y) du}, \quad\quad \tilde{g}(u,y) &=& 
e^{-\frac12 |y-Hu|^2_\Gamma},
\label{eq:g} 
\end{eqnarray}
with shorthand $g_j(u)=g(u,y_j)$.}
Furthermore define $\cC_j$ as the nonlinear operator which 
updates the density according to the $j^{th}$ observation, i.e. 
for $u \sim \hat{\rho}$, $u|y_j \sim \rho$, where
\begin{equation}
\rho = \cC_j \hat{\rho} :=  \frac{ \hat{\rho} g_j} {\int \hat{\rho}g_j}.
\label{eq:update}
\end{equation}
Notice that this operation is treacherous because either large or small values of 
$e^{-\frac12 |y-Hu|^2_\Gamma}$ can cause large error growth.  In particular, $e^{-\frac12 |y-Hu|^2_\Gamma}$
is bounded below only by zero, although it obtains arbitrarily small values arbitrarily rarely, {so the 
probability of the denominator being very small and leading to large amplification of errors is not high.}

Denote the filtering density given $j$ observations by $\rho_j$, and 
then the recursion may be given by
\[\;\;\mbox{Forecast}\;
\begin{array}{l}
  \hat{\rho}_j = \cK^\top \rho_{j-1}, 
 \end{array}                              
\] 
\[\mbox{Update}\;\;\;
\begin{array}{llll}
\rho_j = \cC_j \hat{\rho}_j,
\end{array}
                                     \]  
where $\cK$ is defined in \eqref{eq:dtf2}. %eq:model2}.                                    
Or, in other words $\rho_j = \cC_j \cK^\top \rho_{j-1}$.  In the following section we will derive a deterministic 
solution approach to approximating this recursion.

\subsection{Optimal filtering}\label{ssec:optimal}

{
In the sense of mean-square error,
the optimal point estimator $\hat{u}_j(Y_j)$ 
(as a function of the observations) 
of the signal $u_j$ is $\bbE(u_j|Y_j)$ \cite{jaz70, bobrowski2005functional,
%luenberger1968optimization, BC09, 
pajonk2013sampling, ernst2014bayesian}.  
In other words
\kl{\begin{equation}
\bbE(u_j|Y_j) = {\rm argmin}_{\{\hat{u} = \phi(Y_j); \phi\in L^2\}} \bbE | \hat{u} - u_j |^2,
\label{eq:mmse}
\end{equation}
where the expectation is with respect to $(u_j, Y_j)$, and $L^2$ here denotes
the collection of functions $f$ such that $\bbE_Y ( f^2(Y) ),$ where $\bbE_Y$ denotes
expectation with respect to $Y$.
A short concise proof of this fact may be found in Theorem 5.3 of \cite{jaz70}.
The scrupulous reader may find more satisfaction in the exposition of 
\cite{bobrowski2005functional} Theorem 3.2.6, a consequence of the Doob-Dynkin Lemma 2.1.24.
Note $u_j\in L^2$ under mild assumptions on the kernel $\cK$ and initial distribution \cite{Oks98},
which will indeed be made later on.}
%This 
\todo{The optimal point estimator} is a random variable for $Y_j$ random, and a deterministic variable for 
a given realization of $Y_j$.
Our aim is to solve the true filtering problem, i.e. obtain the full filtering
distribution of $u_j|Y_j$.  It is nonetheless good to know that the optimal 
point estimator may be easily obtained from the true filtering distribution.

Naturally the fidelity with which we approximate $u_j|Y_j$ 
will dictate the fidelity with which we approximate \eqref{eq:mmse}.  
Assume one has access to the random variable $u_j|Y_{j-1}$ in the
complete distributional sense.
Then, upon conditioning the equation \eqref{eq:dtfo}
on $Y_{j-1}$, \kl{one %notices 
finds that $y_j|Y_{j-1} = H u_j|Y_{j-1} + \eta_j$. 
%it is clear that
Therefore \eqref{eq:mmse}
can also be represented in terms of the one-step optimal
point estimator, {following Doob-Dynkin}, where
\begin{equation}
\bbE(u_j|Y_j) =\bbE[ (u_j|Y_{j-1})| (y_j|Y_{j-1}) ] 
= {\rm argmin}_{\{\hat{u} = \phi(y_j|Y_{j-1}); \phi\in L^2%(\bbR^m;\bbR^d)
\}} 
\bbE | \hat{u} - (u_j|Y_{j-1}) |^2.
\label{eq:mmse1}
\end{equation}}
In other words, the minimum mean-square error estimator of the time $j$ state given the 
time $j-1$ filtering distribution and the $j^{th}$ observation is the expectation of the 
time $j$ filtering distribution, as expected.  
This of course requires knowing the full (filtering) distribution of $u_{j-1}|Y_{j-1}$ 
to get the (forecast) distribution of $u_{j}|Y_{j-1}$.  
In fact, one finds that the formula for $\bbE(u_j|Y_j)$ only relies on 
$\bbE(u_j|Y_{j-1})$ in the linear case (c.f. Thm. 3, p. 92 \cite{luenberger1968optimization}).
% so it may not be very useful, but it is nonetheless instructive in devising useful and 
%computationally tractable approximations.

\subsection{Sub-optimal filtering}\label{ssec:sof}

The optimal filter is often very difficult to obtain, particularly in high-dimensional 
problems.  So in practice one may resort to sub-optimal filters.  
We present the one-step optimal linear filter for general models, which only uniquely defines mean
and covariance.
%, and therefore has a natural connection to Gaussians.  

\subsubsection{One-step optimal linear filter}\label{sssec:sof1}

Given the distribution of $u_{j}|Y_{j-1}$, 
obtaining even the one-step optimal point estimator \eqref{eq:mmse1} is often a 
formidable task for high-dimensional models, so one may consider instead the 
{\it one-step optimal linear} point estimator %. %, i.e. 
%In other words, the 
defined as the best estimator $\hat{u}=\phi(y_j|Y_{j-1})$ 
out of the class of linear functions $\phi(y) = K y + b$ (Thm. 1, p. 87 \cite{luenberger1968optimization}).
In the remainder of discussion here all random variables are conditioned on 
$Y_{j-1}$ and the conditional dependence is omitted to avoid notational clutter.
The following optimization problem is solved 
\begin{equation}
m_j(y_j) = {\rm argmin}_{\left \{\hat{u} = \phi(y_j); \phi(y) ~\todo{\rm linear} %\in {\rm span}\{1,y\} 
\right \}} \bbE | \hat{u} - u_j |^2.
\label{eq:mmse1l}
\end{equation}
%where $u_j, y_j$ are conditioned on $Y_{j-1}$.  
\todo{Let $\phi(y) = K_j y + b_j$, and let $y_j=Hu_j+\eta_j$ 
be the random variable defined in \eqref{eq:dtfo}.  Optimizing} 
this equation with respect to $K_j$ and $b_j$ gives \cite{luenberger1968optimization}
\begin{align}
\label{eq:kay}
K_j & = \bbE[(u_j-\bbE u_j) \otimes (y_j - \bbE y_j)] \bbE[ (y_j -\bbE y_j) \otimes (y_j-\bbE y_j) ]^{-1}, \\
\label{eq:mmse1sol}
m_j(y_j) & = \bbE u_j + K_j (y_j - H \bbE u_j).
\end{align}
Notice that 
$\bbE m_j = \bbE u_j$, but indeed it is not necessarily the case that $m_j = \bbE(u_j|y_j)$.  
In fact, this is only the case when $u_j$ is Gaussian.  Note that this holds only while $y_j$
is a random variable.  In practice, for the {\it quenched} case, in which $y_j$ is a deterministic 
{\it realization}, $m_j$ is deterministic as well and is therefore its own expected value.

\todo{Suppose we update the random variable itself, %in Eq. \eqref{eq:mmse1sol}
%rather than its mean as the equation dictates, 
using the single observed {\it realization} 
(non-random) $y_j$, as follows:  
\begin{equation}
v_j = u_j + m_j(y_j) - m_j(H u_j + \tilde{\eta}_j) = u_j + K_j ((y_j-\tilde{\eta}_j) - H u_j), 
%K_j (y_j - H u_j).
\label{eq:mmseenkf}
\end{equation}
%\todo{
where $\tilde{\eta}_j \sim N(0,\Gamma)$.
%Due to linearity and the fact that $\bbE y_j = y_j$ (since it is deterministic), %we have 
One then has that $\bbE v_j = m_j(y_j)$, the
one-step optimal linear estimator for this given $y_j$, %!
%and the covariance of $v_j$ is   
%However, the covariance of $v_j$ is smaller than 
%the error covariance: 
and
%$C_j 
$$\bbE[v_j - m_j(y_j)] \otimes [v_j-m_j(y_j)] = \bbE[u_j - m_j(H u_j + \tilde{\eta_j})]\otimes[u_j-m_j(H u_j + \tilde{\eta_j})],$$ 
where the quantity on the right-hand side is referred to as the error covariance.
In the data assimilation literature, $y_j \pm \tilde{\eta}_j$ is referred to 
as a ``perturbed observation"  \cite{burgers1998analysis}.}
%by an additive term $K_j \Gamma K_j^\top$.  
%This was discovered in the early development of the ensemble Kalman filter \cite{burgers1998analysis} 
%and the solution was to add an independent random variable $\eta_j \sim N(0,\Gamma)$ to $y_j$
%so that the covariance of $v_j$ is given by $C_j$. % in Eq. \eqref{eq:mmse1eco}.  
The material in this section is
also discussed in the recent
works \cite{pajonk2012deterministic, ernst2014bayesian, pajonk2013sampling, pajonk2012stochastic}. }

\section{Ensemble Kalman filter}\label{sec:enkf}

The EnKF in principle can be viewed as an attempt 
to construct a suboptimal filter which is optimal among the class of all filters 
in which the update is given by linear transformation of the observation.  
Such filters follow in principle from the procedure outlined in 
Sec. \ref{sssec:sof1}, at least for a single observation update.  
%As discussed there, such filter is unique only
%as far as its mean being defined by \eqref{eq:mmse1sol} and its covariance
%by \eqref{eq:mmse1eco}, which means there are an infinite number of candidates.  
%Another natural 
One choice might be the one that is completely defined by its mean and 
covariance, i.e. the corresponding Gaussian.  In what follows, we will see that this
is not the optimal one and its error is larger than the EnKF.  
First the mean-field EnKF equations are presented, following the discussion 
of the previous section.  Then the standard finite-ensemble EnKF 
is presented.

\subsection{Mean-field limit}\label{ssec:mfenkf}

{The term {\it mean-field} typically refers to the empirical measure
$\frac1N \sum_{n=1}^N \delta_{v^{(i)}_{j}}$ (the mean of the Dirac point masses)
of a system of interacting particles $\{v^{(i)}_j\}_{n=1}^N$.  Mean-field {\it interactions} refer
to interactions of the mean-field with the individual particles, and the mean-field 
{\it limit} is the measure $\rho_j = \lim_{N\rightarrow \infty} \frac1N \sum_{n=1}^N \delta_{v^{(i)}_{j}}$,
assuming it exists.  The term mean-field limit will also be used to describe
the corresponding limiting system, in which the 
\kl{particles are i.i.d. but an individual depends on the statistics of its distribution.}
%"interactions" depend only on the statistics
%of the individual or, equivalently, the interaction with the infinite particle system.  
Such system is completely defined by a single process, 
which will be referred to as a mean-field {\it process}.}

Beginning with an approximation of the 
forecast random variable $\widehat{v}_j \approx u_j|Y_{j-1}$, 
one may construct a suboptimal filter update 
$v_j \approx u_j|Y_{j}$ using Eq. \eqref{eq:mmseenkf}, 
%by modifying the observations $y_j \rightarrow y_j + \eta_j,$ 
%where $\eta_j \sim N(0,\Gamma)$ is an independent random variable, 
%in order to inflate the covariance of $v_j$ 
\todo{as discussed in the end of Section \ref{sssec:sof1}.} 
%so that the covariance of $v_j$ is given by Eq. \eqref{eq:mmse1eco}.  
This is the mean-field limiting 
interpretation of the so-called perturbed observation EnKF from the literature \cite{burgers1998analysis},
and will be the EnKF algorithm we focus on here.
\kl{Alternative presentations of this material are available in \cite{ernst2014bayesian}, %Sec. 7.3 and Theorem 7.2, 
and %an alternative philosophical perspective is presented in 
the works
\cite{pajonk2012deterministic, pajonk2013sampling, pajonk2012stochastic}. }

The following mean-field process defines the mean-field limiting EnKF (MFEnKF):

\[\;\;\;\;\mbox{Forecast}\;\left\{\begin{array}{lll}
    \widehat{v}_{j+1}&\sim\mathcal{K}(v_j, \cdot), \\%\vspace{4pt}\\
\hm_{j+1}&=\E\widehat{v}_{j+1},\\%\vspace{4pt}\\
\hc_{j+1}&=\E(\widehat{v}_{j+1}-\hm_{j+1})\otimes(\widehat{v}_{j+1}-\hm_{j+1})
 \end{array}\right.
                                   \] 
\[\mbox{Update}\;\;\;\;\left\{\begin{array}{llll} S_{j+1}&=H\hc_{j+1}H^T+\Gamma \\%\vspace{4pt}\\ 
\label{eq:anal}
K_{j+1}&=\hc_{j+1}H^TS_{j+1}^{-1}\\
%\vspace{4pt}\\
{\tilde y}_{j+1}&=y_{j+1}+\tilde{\eta}_{j+1}\\
v_{j+1}&=(I-K_{j+1}H)\widehat{v}_{j+1}+K_{j+1}{\tilde y}_{j+1}.%\vspace{4pt}\\
\end{array}\right.
                                     \] 
Here 
$\tilde{\eta}_j$ are i.i.d. draws 
from 
$\G(0,\Gamma)$  and perturbed observation 
refers to the fact that the update sees an observation perturbed 
by an independent draw from $\G(0,\Gamma)$.

Notice that the one-step optimal linear  filter $m_{j+1}=\E{v}_{j+1}$ and its covariance 
$C_{j+1}=\E({v}_{j+1}-m_{j+1})\otimes({v}_{j+1}-m_{j+1})$
are precisely equal to the mean and covariance of the posterior distribution under the assumption 
that the prior forecast distribution of $\widehat{v}_{j+1}$ is Gaussian with mean given 
by its mean $\hm_{j+1}$ and covariance given by its covariance $\hc_{j+1}$.  
However, under this assumption the posterior is also Gaussian, while the 
MFEnKF distribution is not.

\subsection{Finite ensemble}\label{ssec:enkf}

The standard EnKF in practice consists of propagating an ensemble, using this ensemble to estimate the covariance and mean, 
and then following the procedure described in the previous section.  
The EnKF is executed in a variety of ways and we consider here only one of these, the perturbed observation EnKF.
It is given as a Monte Carlo approximation of the MFEnKF version: 

\[\;\;\;\;\;\;\;\quad\quad\mbox{Forecast}\;\left\{\begin{array}{lll}  
\widehat{v}_{j+1}^{(n)}&\sim\mathcal{K}(v_j^{(n)}, \cdot), \;n=1,...,N\\%\vspace{4pt}\\
\hm_{j+1}&=\frac1N\sum_{n=1}^N\widehat{v}_{j+1}^{(n)},\\%\vspace{4pt}\\
\hc_{j+1}&=\frac1N\sum_{n=1}^N(\widehat{v}_{j+1}^{(n)}-\hm_{j+1})\otimes(\widehat{v}_{j+1}^{(n)}-\hm_{j+1})
 \end{array}\right.
                                   \] 
\[\mbox{Analysis}\;\;\;\;\left\{\begin{array}{ll} 
v_{j+1}^{(n)}&=(I-K_{j+1}H)\widehat{v}_{j+1}^{(n)}+K_{j+1}y_{j+1}^{(n)}%\vspace{4pt}
\\y_{j+1}^{(n)}&=y_{j+1}+\tilde\eta_{j+1}^{(n)}                                    \end{array}\right.
                                     \] 
Here $\tilde\eta_j^{(n)}$ are i.i.d. draws from 
$\G(0,\Gamma)$, and $K_{j+1}$ is defined as in the previous section.

Analysis of the finite-ensemble case is more involved because of the correlation
between ensemble members arising from the sample covariance.  This issue is explored in 
detail in \cite{le2011large} and it is shown that the finite-ensemble EnKF converges asymptotically
to the mean-field version for models of the form \eqref{eq:model}.  
We will show in the subsequent section that this proof may be extended to the case \eqref{eq:dtf2}.

\section{Fokker-Planck filters}
\label{sec:fpf}

We consider here the setting in which \eqref{eq:dtf2} is given by the
solution of an SDE over a fixed interval of time $h$.  We will solve
the filtering problem by
approximating the evolution of the density with the Fokker-Planck
equation between observations, and then using a variety of different 
approximations in the update.  Different updates will yield accurate 
approximations of (i) the true filtering distribution, (ii) the mean-field EnKF distribution 
as well as (iii)/(iv) two different Gaussian approximations.

\subsection{The setup}
\label{ssec:fpfsetup}

We will consider the following general form of stochastic process 
$u : \bbR_+ \rightarrow \bbR^d$:
\begin{equation}
du = F(u) dt + \sqrt{2 b} dW, \quad u(0)=u_0,
\label{eq:langevin}
\end{equation}
where $dW$ is the increment of a Brownian motion $W \in \bbR_+ \times \bbR^d$, 
$F:\bbR^d \rightarrow \bbR^d$ 
is differentiable and Lipschitz.
and $b \in (0,\infty)$ constant.
\footnote{This can easily be extended to general nonnegative symmetric 
operator $b(u):\bbR^d \rightarrow \bbR^{d \times d}$, 
but constant scalar $b>0$ will be sufficient here and will simplify the discussion.}

It is well-known that the pathspace distribution 
of solutions to Eq. \eqref{eq:langevin} over realizations
of $W(t)$ has density $\rho: \bbR^d \times \bbR_+ \rightarrow  \bbR_+$ 
given by the solution of the Fokker-Planck equation
\begin{equation}
\partial_t \rho = \cG \rho, \quad \cG \rho = \nabla \cdot (b \nabla \rho - F \rho),
\quad \rho(u,0)=\delta(u_0-u),
\label{eq:fok1}
\end{equation}
with zero boundary conditions at $\pm \infty$.
In other words, for $\varphi:\bbR^d \rightarrow \bbR$ we have
$$
\bbE_{u(t)}[\varphi(u(t))] = \int_{\bbR^d} \varphi(u) \rho(u,t) du.
$$
The reader is referred to the works \cite{evans, Marko_Vill00} and references
therein for results and estimates regarding the regularity of Fokker-Planck equations.  
Sufficient regularity will be assumed here, and will not be dealt with further.  
The reader is referred also to \cite{Ris84} for a survey of methods of solution and 
applications.

The solution to equation \eqref{eq:fok1} will be approximated by a finite-dimensional 
vector.  
For now, it will be taken as an assumption that we can approximate the 
solution to this equation in $d$ dimensions using \kl{approximately} $N$ degrees of freedom with
accuracy $\cO(N^{-\kappa_1/d})$ -- this may be obtained by merely constructing a
 tensor-product grid \kl{with ceil$(N^{1/d})$ points in each dimension, where ceil$(x)$ 
 denotes the smallest integer greater than or equal to $x$,} and using a 
$1d$ deterministic approximation method of order $\kappa_1$ in each dimension.   
 Time discretization error will be ignored in the present work to avoid clutter, 
 but can be easily included later.  
 
 The numerical approximation of the Fokker-Planck equation between observation
 times will be combined with updates to the corresponding density at observation
 times using a quadrature rule to approximate \eqref{eq:update}.  
 Similarly to above, it will be assumed
 that we have a quadrature rule of order $\kappa_2$.  When combined, this gives a method
 of order $\kappa = {\rm min}\{\kappa_1,\kappa_2\}$ for the filtering density for $d=1$, 
 or $\cO(N^{-\kappa/d})$ for the $d-$dimensional problem.    
 This will be the basis of the following algorithms.

\subsection{The algorithms}
\label{ssec:algorithms}

Several filtering algorithms based on accurate solution of the Fokker-Planck equation
are proposed in this section for comparison with standard EnKF.  
It is important to emphasize the pedagogical benchmark nature of this work, which
is intended to elucidate various aspects of filtering and inspire development of new 
usable algorithms, rather than to propose new usable algorithms itself. 

First, for use as a benchmark, 
we present the deterministic approximation of the 
full Fokker-Planck filter (FPF) which targets 
the true filtering density, i.e. is asymptotically unbiased, 
and hence recovers a consistent approximation of 
the optimal filter as its expectation. 
Next, we present the deterministic approximation of the 
mean-field EnKF %approximation 
in density form using Fokker-Planck evolution of the density (DMFEnKF).
Finally, in order to examine the value of the non-Gaussian 
component retained in the MFEnKF distribution for nonlinear
non-Gaussian models, we 
present two approximate filters which 
{\it impose} Gaussianity on the updated (MFEnKF-G1)
and forecast %predicting 
(MFEnKF-G2) densities.

\subsubsection{Full FPF}\label{sssec:ffpf}

In this section the full Fokker-Planck filter is described,
which 
is a consistent approximation of
the true filtering distribution.
First,  
the deterministic evolution 
of the density is approximated using  
a discrete approximation, e.g. finite differences or similar,  
and then the update Eq. \eqref{eq:update} is approximated 
using a quadrature rule, such as the trapezoidal rule 
or similar.  
In general the algorithm is given by the following
\begin{algorithm}[h!]
\title{Full FPF}
\center{Full FPF}
\label{fpf}
\begin{itemize}
\item(1) Approximate the density at time $j$ over space using some 
accurate and economical spatial discretization. 
\item(2) Evolve forward the Fokker-Planck equation for this density 
using an accurate time-stepper, 
obtaining an estimate of the forecast 
distribution at time $j+1$. 
\item(3) Approximate the updated distribution using some integration rule to 
normalize the prior-weighted likelihood, and return to step (2).
\end{itemize}
\end{algorithm}

\kl{
\subsubsection{Deterministic mean-field EnKF}\label{sssec:mfenkf}

In this section the approximation of 
MFEnKF from Sec. \ref{ssec:mfenkf} is derived in 
density form, so that the algorithm may utilize the Fokker-Planck forecast 
density as in the full Fokker-Planck algorithm of the previous section.
The resulting algorithm is called the deterministic MFEnKF (DMFEnKF).

First, consider a random variable $z=x+U \xi$ where $x \in \bbR^d$ and $\xi \in \bbR^m$
are random variables and $U \in \bbR^{d \times m}$ is a matrix consisting of orthonormal 
columns.  Let $U_\perp$ denote and orthonormal basis for the nullspace of $U^\top$, 
so that $R=[U,U_\perp]^\top$ is a rotation on $\bbR^d$.  
Now, let $\tilde{z}=Rz$ and $\tilde{x} = R x$, so that $\tilde{z} = \tilde{x} + [\xi^\top, {\bf 0}]^\top$, 
where ${\bf 0}$ denotes a $d-m$ dimensional row vector.
Denote the densities of $\xi$, $x$, $z$, $\tilde{x}$, and $\tilde{z}$ by $\rho_{\xi}$, $\rho_{x}$, 
$\rho_{z}$, $\rho_{\tilde{x}}$, and $\rho_{\tilde{z}}$, and notice that 
$\rho_{\tilde{x}}(\cdot) = \rho_{x}(R^\top \cdot)$ and $\rho_{z}(\cdot ) = \rho_{\tilde{z}}(R \cdot)$.  
Now, observe that a simple change of variables with unit Jacobian yields
\begin{equation}
\rho_{\tilde{z}}(\tilde{z}) = \int_{\xi \in \bbR^m} \rho_{\tilde{x}}(\tilde{z} - [\xi^\top,{\bf 0}]^\top) \rho_{\xi}(\xi) d\xi.
\label{eq:sumconv}
\end{equation}

The update formula $v_{j-1} \mapsto v_j$ of subsection \ref{ssec:mfenkf} culminates in the 
addition of two independent random variables 
$v_{j} = (I - K_j H) \widehat{v}_j  +  K_j \tilde{y}_j$, where 
$\tilde{y}_j \sim N(y_{j}, \Gamma)$.
Define (for each $j$) the singular value decomposition
of the Kalman gain \eqref{eq:kay} $K_j = U_j \Sigma_j V_j^\top$ where $\Sigma_j \in \bbR^{m \times m}$
is diagonal and positive definite, and $U_j \in \bbR^{d\times m}$ and 
$V_j \in \bbR^{m\times m}$ have orthonormal columns. 
%(recall $d$ is the state dimension and $m$ is the observation dimension).
Defining $z=v_j$, $x=(I - K_j H) \widehat{v}_j$, $\xi = \Sigma_j V_j^\top \tilde{y}_j$,
and $U=U_j$, this is exactly the setting of the previous paragraph.  
The density of $v_j$ is therefore defined as in \eqref{eq:sumconv}. 
 %Let $R_j$ denote the corresponding rotation.
%It is obvious that 
From its definition, %it is obvious that 
the density of $\xi$ is given by
\begin{eqnarray}
\label{eq:dmfenkfg}
\tilde{g}_j(\xi) &=& \exp\left \{ -\frac12 \left | \xi - \Sigma_j V_j^\top y_j \right |_{\Sigma_j V_j^\top \Gamma V_j \Sigma_j}^2 \right \} \quad 
\forall \xi \in \bbR^m,\\
\hat{g}_j(\xi) &=& \frac{\tilde{g}_j(\xi,y_j)}{\int_{\xi \in \bbR^m} \tilde{g}_j(\xi,y_j)},
\label{eq:gk}
\end{eqnarray}
while the density of $x$ arises by a simple change of variables formula.
Finally, define $\cC^G_j$ by its action on a density $\pl$ as follows
\begin{equation}
\begin{array}{llllll}
\label{eq:gupdate1}
\rho_x(x) & =  %\rho_{\tilde{x}}
p((I-K_jH)^{-1}x) {\rm det}[I-K_jH]^{-1} \\ %(=p_{X_1}(u)) \\
(\cC^G_j p)(u) &:= \int_{\xi\in\bbR^m} %\hat{\pl}
\rho_x(u-U_j \xi) \hat{g}_j(\xi) d\xi. 
\end{array}
\end{equation}
This follows %exactly from the above discussion in a precise and straightforward manner, 
precisely from above, observing that $U \xi = R^\top[\xi^\top,{\bf 0}]^\top$ 
and changing variables of \eqref{eq:sumconv} again.

The MFEnKF is therefore given in density form by $\ml_{j} = \cC_{j}^G \cK^\top \ml_{j-1}$,
or its discrete approximation.  In general the DMFEnKF algorithm is given by the 
following
\begin{algorithm}[h!]
\label{mfenkf}
\title{DMFEnKF}
\center{DMFEnKF}
\begin{itemize}
\item(1), (2) Same as in Algorithm Full FPF. %above.
\item(3) Approximate the mean and covariance of the 
{forecast} distribution using an integration rule.
\item(4) Perform the linear change of variables of the predicting density
to the updated variables $u \rightarrow (I-K_jH)u$ using interpolation and 
either rescale by the determinant det$[I-K_jH]$ (or renormalize using an integration rule). 
\item(5) Incorporate the observation via convolution with the density of 
$K_j N(y_j,\Gamma)$ on the range-space of $K_j$, as described above \eqref{eq:gupdate1}.
\end{itemize}
\end{algorithm}

%Some comments are in order.  First, i
Note that in order to \kl{evaluate} $\rho_x(u_i)$ in Eq. \eqref{eq:gupdate1} 
for a point $u_i$ on the grid, it is necessary to find the value 
$p((I-K_jH)^{-1}u_i)$.
 Indeed the transformation is contractive on the 
subspace $UU^\top\bbR^d$ since $|U^\top(I-K_jH)| \leq 1$, 
so $(I-K_jH)^{-1}u_i$ will sometimes lie outside the \kl{original numerical 
domain used for simulation}
\footnote{Here a single fixed grid is assumed, however in practice it may be beneficial,
and even necessary, to allow the grid to adapt with time.}.
In this case one must have $\pl((I-K_jH)^{-1}u_i)\approx 0$, and the 
value is simply set to zero.  
Of course, if the approximation is going smoothly then the density
is approximately zero there, but care must be taken.  
Aside from requiring the value of the original density 
off the grid, the density also needs to be over-resolved in general so
that when it is contracted, and the mesh is effectively coarsened, the resolution is 
still reasonably good.  
%Second, convolution is an expensive operation that should be 
%avoided if possible.  
%It can be avoided by using a square-root filter version, but this direction is not 
%pursued further here.
}

\subsubsection{Mean-field EnKF, with Gaussian approximation}\label{sssec:mfenkfg}

It will be instructive when executing the numerical experiments in Sec. \ref{sec:numerics} 
to also compare the filters above and the EnKF with filters which actually make an {\it explicit} 
Gaussian approximation.
We therefore consider two cases.  
In the first case, we proceed as in the full FPF, but {\it after the update} we retain only the 
best Gaussian approximation, i.e. the Gaussian having mean and covariance given by the 
full update.  In the second case, we approximate the {\it forecast} 
distribution by a Gaussian,
and use the update from Sec. \ref{ssec:mfenkf}.  These algorithms are given as follows
\begin{algorithm}[h!]
\label{eq:mfg1}
\title{MFEnKF-G1}
\center{MFEnKF-G1}
\begin{itemize}
\item(1),(2),(3) Same as in Algorithm Full FPF.
\item(4) Approximate the mean and covariance of the {\it updated} distribution using an integration rule.
\item(5) Approximate the updated density by the Gaussian with the mean and covariance from step (4).
\end{itemize}
\end{algorithm}
\begin{algorithm}[h!]
\label{eq:mfg2}
\title{MFEnKF-G2}
\center{MFEnKF-G2}
\begin{itemize}
\item(1),(2) Same as in Algorithm Full FPF.
\item(3) Approximate the mean and covariance of the {\it forecast} 
distribution using an integration rule, and approximate
the forecast density by a Gaussian with this mean and covariance.
\item(4) Update this Gaussian with observations using the closed form Kalman update formula for the mean and covariance.
\item(5) Construct the updated Gaussian density using the mean and covariance from step (4) and return to step (2).
\end{itemize}
\end{algorithm}

\section{Theoretical results}
\label{ssec:theo}

In this section the densities delivered by the 
algorithms in the previous section are theoretically probed.  
It is proven that if the nonlinearity/non-Gaussianity
is ``small enough", then the sampling error dominates and the Fokker-Planck algorithms 
all asymptotically outperform the standard EnKF for a given cost, up to a critical dimension.  
In particular, in this regime the MFEnKF is superior to EnKF for small enough dimension.  
However, for a significant nonlinearity/non-Gaussianity the 
sampling error is obscured by the non-Gaussian error 
and the EnKF may then perform comparably to the MFEnKF even with a small sample size.

It will be convenient to introduce a distance measure between densities,
%following 
inspired by the one defined on p. 6 of \cite{rebeschini2013can}, to make statements about convergence of
the filters.
\kl{Let $\mu:\bbR^d \rightarrow \bbR$ denote a random measure, such as
$\frac1N \sum_{n=1}^N \delta_{x^{(i)}}$, and define the following norm
$$
\|\mu\| = {\rm sup}_{|f|_L \leq 1} \sqrt{ \bbE \left | \int f d\mu\right |^2 },
$$ 
where $f:\bbR^d \rightarrow \bbR$ and 
$|f|_L = |f|_\infty + {\rm sup}_{x \neq x'} \frac{|f(x)-f(x')|}{|x-x'|}$ is the bounded
Lipschitz norm.  The following notation is used 
$\int f d\mu = \int_{\bbR^d} f(u) \mu(du) = \int_{\bbR^d} f(u) \pi(u) du$,  %is used 
when $\mu$ has a density $\pi$.} %, with slight abuse of identical notation for measure and density.}
In this case, we define $\|\pi\| := \|\mu\|$. 
Now define the metric $d(\cdot,\cdot)$
between two random measures with densities
$\ml$ and $\pl$ as follows:
\begin{equation}
d(\ml,\pl) = \|\ml - \pl\|.
\label{eq:metric}
\end{equation}
This metric is relevant for probability densities 
as it measures, in a mean-square sense, the distance between observables, 
i.e. expectations with respect to the given random measures. 
{For example, one may be interested in the mean-square error (MSE) 
of the mean, or the variance, or some other quantity of interest.}   
Here the densities will be the filtering density or approximations to it
and the randomness of the densities 
will come from the vector of past random observations
upon which the filtering density depends, as well as 
the random samples giving rise to the ensemble empirical measure
 in the case of EnKF.  
%\kl{ 
This section is concluded with a comment on the distance measure above.  Notice that 
 $\{|f|_\infty \leq 1\} \supset \{|f|_L \leq 1\}$, so the analogous distance measure over the former
 set of functions, denoted by $d_\infty(\cdot,\cdot)$ \footnote{This is the metric used in \cite{rebeschini2013can}.}, dominates the one used here. 
 %}.  
 Notice that 
 the metric $d_\infty$ is equivalent to total variation when the given measures are not random.
 This space of test functions is ubiquitous in probability theory, as it is dual to the space 
 of finite measures.  
% The proof of convergence of EnKF requires locally Lipschitz functions with at most polynomial 
% growth at infinity.
% Here it will be convenient to work with
% the given intersection of the test functions in order to combine convergence of 
% EnKF to MFEnKF with convergence of MFEnKF to the filtering distribution.

 \subsection{EnKF converges to MFEnKF}
 \label{sssec:enkfmfenkf}
 
As mentioned before, for clarity 
of exposition, 
time integration is assumed to be exact.  {The 
true filtering density at the $j^{th}$ observation time (i.e. $t_j=jh$) 
is denoted $\rho_{t_j}$, 
the MFEnKF density at the $j^{th}$ observation time is denoted $\ml_{t_j}$,
the density approximating $\ml_{t_j}$ with the DMFEnKF filter is denoted $\ml^N_{t_j}$, and 
the standard EnKF distribution is denoted $\hat{\ml}^N_{t_j}$.}
The sub-subscript notation is introduced for clarity because the increment
$h$ will vary.  Also, the kernel $\cK_h$ will replace \eqref{eq:dtf2} to denote the 
dependence on the increment $h$.
\todo{
%(in fact, the results are in $L^p$ and for more general test functions) 
%for the linear Gaussian case and .  
For the linear case, Theorem 1 of \cite{mandel2011convergence} 
shows that for {\it each ensemble member} $v^{N,(n)}_j$ of $\hat{\ml}^N_{t_j}$,
$\bbE|v^{N,(n)}_j - V^{N,(n)}_j|^p \rightarrow 0$, as $N\rightarrow \infty$,  
where $V^{N,(n)}_j$ is an i.i.d. draw of the limiting measure $\ml_{t_j}$, obtained using the 
{\it same realizations} of randomness as $v^{N,(n)}_j$, and dependence of the latter on
ensemble size $N$ is made explicit.  
This result for $p=2$ implies the convergence $d(\hat{\ml}_{t_j}^N, \ml_{t_j}) \rightarrow 0$,
by the triangle inequality, a fact stated explicitly for the unbounded
$f(u)=u$ and $f(u)=uu^\top$ in Corollary 1 of that paper.
For the special case of \eqref{eq:model}, 
Proposition 4.4 of \cite{le2011large} shows that in fact $(\bbE|v^{N,(n)}_j - V^{N,(n)}_j|^p)^{1/p} = \cO(N^{-1/2})$,
and Theorem 5.2 of that paper 
%\cite{mandel2011convergence} and 
uses this to show that %this implies that 
for all $f$ Lipschitz with polynomial growth at infinity, and for all $p$,
\begin{equation}
\left( \bbE \left |\int f (u) (\ml_{t_j}(u) - \hat{\ml}^N_{t_j}(u) ) du \right|^p \right)^{1/p} = \cO(N^{-1/2}),
\label{eq:leglp}
\end{equation}
%$d(\hat{\ml}_{t_j}^N, \ml_{t_j})=\cO(N^{-1/2})$ }
which implies $d(\hat{\ml}_{t_j}^N, \ml_{t_j})=\cO(N^{-1/2})$, 
since the latter is weaker, considering only $p=2$ and bounded $f$.} %$\varphi$.}
%Both papers use the notion of an empirical measure of i.i.d. samples of the limiting 
%measure $\ml_{t_j}$ in order to show convergence.
Assume the following.

\begin{assumption}[EnKF Assumptions]
\begin{itemize}
\item[(i)] For all $a,b \in \bbR^d$ and $u\sim \cK(a,\cdot)$, $v\sim \cK(b,\cdot)$, 
where the driving noise in $\cK$ is the same realization in each case,
$|u-v| \leq C |a-b|$ almost surely.
\item[(ii)] $\bbE |\cK(u,\cdot)|^{p} \leq C'(1+ \bbE |u|^{p})$, for all random variables $u \in \bbR^d$ with 
$\bbE|u|^{p} < \infty$ and for all $p \geq 2$.
\end{itemize}
\label{ass:enkf}
\end{assumption}
Both these conditions are satisfied by models of the form \eqref{eq:langevin} with Lipschitz drift, where the 
constants $C$ and $C'$ depend exponentially on the time increment $h$.  
See e.g. \cite{evans2012introduction} p. 95 for a clear and concise statement for $2p$, with integer $p\geq1$.
Interpolation completes the set of $p$.
Finally, assume $\rho_0=\ml_0$ has finite moments of {\it all} orders $p\geq 2$.  
\kl{The following theorem establishes the convergence of EnKF $\hat{\ml}^N_{t_j}$ 
to MFEnKF $\ml_{t_j}$.}
\
\begin{theorem}
Given Assumptions \ref{ass:enkf}, the following convergence result holds
\begin{equation}
\label{eq:enkfconv}
d(\hat{\ml}^N_{t_j},\ml_{t_j}) = \cO(N^{-1/2}).
\end{equation}
\label{enkftheo}
\end{theorem}
\begin{proof}
\kl{
%The two key points are the Lipschitz condition (i), and the existence of all moments for finite $j$, given the sa
%me for
%$\rho_0$.  The existence of all moments for finite $j$ follows from Assumption
%(ii) above, boundedness of the Kalman gain $|K_{t_j}|$, and the existence of all moments of the observation $y_{t_j}$.  
Notice that Assumptions \ref{ass:enkf} (i) and (ii) are sufficient to establish a priori $L^p$ estimates 
for the signal process, and the mean-field EnKF for all $p\geq 1$, %for finite $j$, 
given the same for $\rho_0$.  One simply replaces Assumption A in
Section 2.2 of \cite{le2011large} and follows exactly the same steps in that section.
The proof of \cite{le2011large} Theorem 5.2 then goes through exactly the same under the above assumptions for  
this broader class of models.  To see the relation to the error metric \eqref{eq:metric}
notice that for bounded Lipschitz $\phi$, one has $|\phi|_L<\infty$.  The collection of normalized
$\phi/|\phi|_L$ comprises the test functions used in \eqref{eq:metric}.  Theorem 5.2 of \cite{le2011large}
establishes the rate of $L^p$ convergence for 
%The proof indeed covers 
not only $p=2$ as in the metric \eqref{eq:metric}, and this class of test functions, 
but indeed for all $p\geq2$, and a more general
class of test functions $\phi$, which need not be bounded and can have polynomial growth at infinity,
as in equation \eqref{eq:leglp}.}
%See the supplementary material for a more detailed sketch of the proof.}
\end{proof}

We will furthermore assume the availability of a deterministic FP method as described 
in the previous section, involving a
numerical discretization of the Fokker-Planck equation and a numerical quadrature rule
for $d=1$ with minimal order of $\kappa$ between the two.  
Let $N$ denotes the 
number of degrees of freedom used in the approximation (evenly divided over the dimensions
as in a tensor product meshgrid for the FP methods).  
\kl{
The following theorem establishes the convergence of DMFEnKF ${\ml}^N_{t_j}$ 
to MFEnKF $\ml_{t_j}$.
Before stating the theorem two more assumptions will be necessary.
\begin{assumption}[DMFEnKF Assumptions]
\begin{itemize}
\item[(i)] The solutions %of the MFEnKF system 
$\ml_{t_j}$, and $\ml^N_{t_j}$ are continuous with respect to the spatial variable;
%(it would be unusual to find a suitable numerical method for a parabolic equation without this property).
\item[(ii)] There exists $\varepsilon>0$ such that 
$\varepsilon \leq \sigma(C_j), \sigma(C^N_j) \leq \varepsilon^{-1}$
for all $j$, where $\sigma(A)$ denotes the singular values of $A$.  
%Furthermore, there is a $\delta>0$ such that $\sigma (K_j), \sigma (K^{N}_j) > \delta$;
%$|K_{t_j}H|<1-\varepsilon$ and $|K^N_{t_j}H|<1-\varepsilon$, 
%where $K_{t_j}$ and $K^N_{t_j}$ are the Kalman gains at time 
%$t_j$ for $\ml_{t_j}$ and ${\ml}^N_{t_j}$, respectively 
%(notice if $H=I$, then it is guaranteed provided $|C_j^N|<\infty$;
%if $H\neq I$, then one also requires that there is some $\delta>0$ such that 
%min$\{\sigma(C_j^N), \sigma(C_j) \} \geq \delta > 0$).
\item[(iii)] The solution of $\ml_{t}$ has the property that $\ml_{t}<C e^{-|x|}$ 
%for some $\alpha\geq1$ 
for all $t$.  Furthermore, 
the domain for both MFEnKF and DMFEnKF are truncated at some $\Omega=\{|x| \leq R\}$, where
%$|\Omega|<R$ 
$R$ is fixed, resulting in a fixed bias (which will be ignored henceforth).  
Similarly the domain of integration of \eqref{eq:gupdate1} becomes $H\Omega \subset \bbR^m$;
\item[(iv)] 
There exists some $\alpha>0$ such that
\begin{equation}
\alpha \leq g(u,y) \leq \alpha^{-1} \quad \text{and} \quad \alpha \leq \hat{g}_{t_j}(\xi,y) \leq \alpha^{-1} \quad \quad
\forall ~u \in \bbR^d, ~ y, \xi \in \bbR^m, %~v\in U_{t_j}U_{t_j}^\top\bbR^d, %~ j \in \Z,
\label{eq:dublin}
\end{equation}
where $g$ is given by Eq. \eqref{eq:g} and $\hat{g}_{t_j}$ is given by \eqref{eq:gk}, respectively.
\end{itemize}
\label{ass:dmfenkf}
\end{assumption}

\begin{remark} The following remarks are in order in connection to the Assumptions above.
\begin{itemize}
\item[(a)] The forward model is assumed to be uniformly elliptic, hence with smooth solution.  
The update is a convolution with a Gaussian, which is again a smoothing operation.  So property (i)
is quite natural for the forward model.
It would be unusual to find a suitable numerical method for such parabolic equation without property (i).
\item[(b)] Notice that the bounds of (ii) are guaranteed for non-trivial (not deterministic) model \eqref{eq:dtf2}
for the mean-field process (see also proof of Theorem \ref{enkftheo}), as only the forward model can lead to 
zero covariance.  In particular, the uniformly elliptic assumption ($b>0$) ensures non-zero covariance. 
Such bounds can be \it{imposed} 
for the finite-ensemble process with an $N$-dependence to ensure they remain below the statistical error.
%Notice if $H=I$, then (ii) is guaranteed provided $|C_j^N|<\infty$;
%if $H\neq I$, then one also requires that there is some $\delta>0$ such that 
%min$\{\sigma(C_j^N), \sigma(C_j) \} \geq \delta > 0$.  This is because 
%Note that $0 < |A(A+B)^{-1}|, |(A+B)^{-1}A|<1$ for all $A, B$ symmetric with $B$ positive definite.  
%If $H=I$, then $A=C_j^{(N)}$ and $B=\Gamma$.  If $H\neq I$,
%then 
%Letting $B=(C_j^{(N)})^{-1}$ and $A=H^\top \Gamma^{-1} H$, one can see immediately that 
This implies the existence of a $\delta>0$ such that 
$|K^{(N)}H| = |((C_j^{(N)})^{-1}+H^\top \Gamma^{-1} H)^{-1}H^\top \Gamma^{-1} H| \leq 1-\delta$. 
 To see this, it suffices to observe that for $A=A^\top>0$ and $B=B^\top\geq 0$,
then $(A+B)^{-1}B = (I+A^{-1}B)^{-1}A^{-1}B$, where $A^{-1}=A^{-\top}>0$, and $A^{-1}B = A^{-1/2}(A^{-1/2}B)$ 
has the same eigenvalues as $A^{-1/2}BA^{-1/2}$, which is again symmetric positive semidefinite. %, with 
So $\sigma( I-K^{(N)}H ) > \delta$.
% \todo{reverify or assume!}.
Furthermore, one has 
$\sigma [ K^{(N)}] = \sigma[ %C_j^{(N)} (  ) ] %
((C_j^{(N)})^{-1}+H^\top \Gamma^{-1} H)^{-1}H^\top \Gamma^{-1} ] 
> \delta$ for some $\delta>0$. 
%, where $\sigma$ is used here to denote singular value.  
This follows from the assumptions that $\Gamma^{-1}$ is non-degenerate and $H$ has rank $m$, 
and the sum of a positive semidefinite and definite matrix is positive definite.   
\item[(c)] The assumption (iii) is made to avoid technical difficulties.  Notice that this induces a fixed bias of 
the size $e^{-R}$.  An $N-$dependent domain can be defined such that 
$R(N) = (\kappa/d + \delta) \log(N)$, for example, so the forward solve still has error 
$\cO(N^{-\kappa/d})$.  But this will mean $|\Omega(N)| = C R(N)^{d} = \cO(\log(N)^d)$, 
which will impact the rate for the DMFEnKF filter with the same log factor.  
Technical difficulties will therefore be avoided with the fixed bias, noting that the results hold only 
up to this bias level. 
\item[(d)] Assumption (iv) commonly appears in theoretical results for particle filters 
\cite{rebeschini2013can, BC09,  lsz2015, del2001stability}.
Such condition is typically necessary to prove convergence of filtering algorithms
due to the nonlinearity of the update. For observed $y$, Assumption (iii) implies (iv).  
%(in particular, the latter is required to ensure the stability
%of the constant in the second statement of \eqref{eq:erdisc}), 
%and it is rarely strictly satisfied in practice.  
%However, such conditions are often satisfied with a high probability, and furthermore 
%the filters often do work in practice despite not strictly satisfying such assumption. 
%This suggests more careful analysis is needed.  
%\kl{For the sake of carrying through the results here, 
%it suffices to consider the compactification of the domain at some sufficiently large radius 
%(say double the finest computational domain) so that \eqref{eq:dublin} holds.}
\item[(e)] Note the DMFEnKF method is not random, so $d_\infty$ becomes total variation
distance here, for observed $y$.
\end{itemize}
%the norm \eqref{}
\label{rem:dmfenkfass}
\end{remark}

\begin{theorem}
The following convergence result holds
\begin{equation}
d_\infty \left (\ml^N_{t_j},\ml_{t_j} \right) = \cO\left(%\frac{1}
{N^{-\kappa/d}}\right). 
\label{eq:erdisc}
\end{equation} 
\label{thm:dmfenkf}
\end{theorem}
\begin{proof} 
The proof is by induction and holds trivially for $t_0=0$.
Assume it holds at step $t_{j-1}$, so 
$d_\infty\left (\ml^N_{t_{j-1}},\ml_{t_{j-1}} \right) = \cO(N^{-\kappa/d})$.  
Let $\cK^N$ denote the propagation through the deterministic FP approximation
with $N$ spatial degrees of freedom.  The triangle inequality provides
\[
\begin{split}
d_\infty\left ((\cK^N)^\top \ml^N_{t_{j-1}}, \cK^\top \ml_{t_{j-1}} \right) & \leq 
d_\infty\left ((\cK^N)^\top \ml^N_{t_{j-1}}, \cK^\top \ml^N_{t_{j-1}} \right) \\ & + 
d_\infty\left (\cK^\top \ml^N_{t_{j-1}}, \cK^\top \ml_{t_{j-1}} \right) \lesssim N^{-\kappa/d},
\end{split}
\]
where the first term is controlled by the assumed property of the kernel $\cK^N$
and the second term is controlled by the induction hypothesis 
and the standard property of Markov kernels (see for example Lemma 4.9 of 
\cite{lsz2015}).
Denote $\tilde{\ml}_{t_{j}}^N := (\cK^N)^\top \ml^N_{t_{j-1}}$ 
and $\tilde{\ml}_{t_{j}} := \cK^\top \ml_{t_{j-1}}$, so we have 
$d_\infty\left (\tilde{\ml}^N_{t_{j}}, \tilde{\ml}_{t_{j}} \right) \lesssim N^{-\kappa/d}$.

Now, the update is a bit more technical.  
One must first note the continuity of $K$ as a function of $C$.  This may be found for example in
Proposition 3.1 of \cite{le2011large} or Lemma 5.1 of \cite{kwiatkowski2015convergence}.  
For $K^{(')} = C^{(')} H^\top (H C^{(')} H^\top +\Gamma)^{-1}$ one has 
\begin{equation}
\label{eq:continuity}
|K-K'| \lesssim |C-C'|.
\end{equation}

%Notice that, by virtue of the fact that $\tilde{\ml}^{N}_{t_{j}}$ is defined only on
%some numerical domain $\Omega^N$, one has that
%$|\tilde{\ml}_{t_{j}} - \tilde{\ml}^{N}_{t_{j}}| = 
%|\tilde{\ml}_{t_{j}}{\mathbf 1}_{\Omega^N} - \tilde{\ml}^{N}_{t_{j}}| + 
%|\tilde{\ml}_{t_{j}}{\mathbf 1}_{\bbR \setminus \Omega^N}|,$
%with ${\mathbf 1}_{A}(x)=\{1, x\in A, 0 ~{\rm else}\}$ the indicator function
%on the set $A$.  
%Hence, both terms share 
%the same bound $\cO(N^{-\kappa/d})$.

Recall Assumption \ref{ass:dmfenkf}(iii) and Remark \ref{rem:dmfenkfass}.
The Frobenius norm bounds the spectral norm, and, since %\newline
$(\tilde{C}^{(N)}_{t_j})_{ik} = \int_{\bbR^d} u_iu_k \tilde{\ml}^{(N)}_{t_{j}}(u) du$ 
 (assuming WLOG $\bbE^{(N)}u_i=0$ for all $i$ for notational simplicity alone)
for all $i,k=1,\dots,d$, one has 
\[\begin{split}
|(\tilde{C}^{N}_{t_j})_{ik} - (\tilde{C}_{t_j})_{ik} | & = 
\int_{\Omega} u_iu_k [ \tilde{\ml}^{N}_{t_{j}}(u) -  \tilde{\ml}_{t_{j}}(u) ] du %\\
%& 
\lesssim d_\infty (\tilde{\ml}^{N}_{t_{j}},\tilde{\ml}_{t_{j}})
% + {\color{red} \int_{\bbR^d \setminus \Omega^N} u_i u_k \tilde{\ml}_{t_{j}}(u) du}
\lesssim N^{-\kappa/d}.
\end{split}
\]
Therefore, 
\begin{equation}
|K^N_{t_j}-K_{t_j}| \lesssim |\tilde{C}^{N}_{t_j} - \tilde{C}_{t_j}| = \cO(N^{-\kappa/d}).
\label{eq:bndk}
\end{equation}

Now it will be shown that 
%\[\begin{split}
\begin{eqnarray}
\nonumber
d_\infty\Big(  & {\rm det}[I-K^N_{t_j}H]^{-1}\tilde{\ml}^N_{t_{j}}((I-K^N_{t_j}H)^{-1}~\cdot~), \\
& {\rm det}[I-K_{t_j}H]^{-1} \tilde{\ml}_{t_{j}}((I-K_{t_j}H)^{-1}~\cdot~) \Big)
 \lesssim N^{-\kappa/d}. \, 
\label{eq:preup}
 \end{eqnarray}
 %\end{split}
%\]
This can be decomposed into a sum of three terms, which will be handled individually.

First term: 
\[\begin{split}
d_\infty\Big(  & {\rm det}[I-K^N_{t_j}H]^{-1}\tilde{\ml}_{t_{j}}((I-K_{t_j}H)^{-1}~\cdot~), \\
& {\rm det}[I-K_{t_j}H]^{-1} \tilde{\ml}_{t_{j}}((I-K_{t_j}H)^{-1}~\cdot~) \Big) \\ \lesssim 
& |{\rm det}[I-K^N_{t_j}H]^{-1} - {\rm det}[I-K_{t_j}H]^{-1}| \\ \lesssim 
& {\rm det}[I-K_{t_j}H]^{-1} | 1 - {\rm det}[I-K^N_{t_j}H]{\rm det}[I-K_{t_j}H]^{-1} | 
%\\ \lesssim
%&  N^{-\kappa(d-1)/d} < 
\lesssim \cO(N^{-\kappa/d}).
\end{split}\]
The last line follows from the fact that ${\rm det}[I-K_{t_j}H]^{-1}<\infty$ by Assumpiton (ii), 
the fact that $|I-[I-K^N_{t_j}H][I-K_{t_j}H]^{-1}| \leq |I-K_{t_j}H|^{-1} |H| |K^N_{t_j}-K_{t_j}| \lesssim N^{-\kappa/d}$ 
by \eqref{eq:bndk}, 
and the fact that if $|I-A| = \varepsilon$, then for all eigenvalues $\lambda_k$ of $A$, one has 
%such that there exists a $v$ with
%$A v_k = \lambda_k v_k$, 
$\lambda_k\leq 1+\varepsilon$, so  
${\rm det}(A) =\prod_{k=1}^d \lambda_k \leq (1+\varepsilon)^d = 1 + \cO(\varepsilon)$. 
%$|1-{\rm det}(A)| $

Second term:
  \[\begin{split}
  d_\infty\Big(  & {\rm det}[I-K^N_{t_j}H]^{-1}\tilde{\ml}_{t_{j}}((I-K^N_{t_j}H)^{-1}~\cdot~), \\
  & {\rm det}[I-K^N_{t_j}H]^{-1} \tilde{\ml}_{t_{j}}((I-K_{t_j}H)^{-1}~\cdot~) \Big) \\
  & \lesssim %\delta^{-1} 
  \int_{\Omega}
  |\tilde{\ml}_{t_{j}}((I-K^N_{t_j}H)^{-1}u) - \tilde{\ml}_{t_{j}}((I-K_{t_j}H)^{-1}u)| du  \\
  & \lesssim %\delta^{-1} 
  |\Omega | |(I-K^N_{t_j}H)^{-1} -  (I-K_{t_j}H)^{-1}|,
  \end{split}\]
  where %$\delta$ is 
  the first bound follows from Assumption \ref{ass:dmfenkf}(ii) (see also the Remark \ref{rem:dmfenkfass}b) 
  and the second follows from the continuity and bounded domain Assumptions \ref{ass:dmfenkf}(i) and (iii).
  Notice that 
  $$
 ( I-K^N_{t_j}H)^{-1} -  (I-K_{t_j}H)^{-1} = (I-K^N_{t_j}H)^{-1}[(K_{t_j} - K^N_{t_j}) H](I-K_{t_j}H)^{-1}.
  $$
  Therefore, the required bound follows from the bound on $|K_{t_j} - K^N_{t_j}|$ \eqref{eq:bndk}, 
  noting again Assumption \ref{ass:dmfenkf}(ii).  
%Since %${\rm det}[I-K^N_{t_j}H]^{-1}<\infty$

Third term: this is simply a change of variables, so
\[\begin{split}
d_\infty\Big(  & {\rm det}[I-K^N_{t_j}H]^{-1}\tilde{\ml}^N_{t_{j}}((I-K^N_{t_j}H)^{-1}~\cdot~), \\
& {\rm det}[I-K^N_{t_j}H]^{-1} \tilde{\ml}_{t_{j}}((I-K^N_{t_j}H)^{-1}~\cdot~) \Big) =
d_\infty (\tilde{\ml}^{N}_{t_{j}},\tilde{\ml}_{t_{j}}).
\end{split}\]

Now consider $\tilde{g}_{t_j}$ and $\hat{g}_{t_j}$ 
defined in \eqref{eq:dmfenkfg} and \eqref{eq:gk}, and 
let $\tilde{g}^N_{t_j}$ and $\hat{g}^N_{t_j}$ denote the DMFEnKF
versions with $N$ degrees of freedom.
For simplicity of notation we will % (and without loss of generality), 
drop the $t_j$.  First consider the case in which 
$K^{(N)}=\Sigma^{(N)}$ is already diagonal and let $\Gamma=I$.
So, by continuity of $f(x)=e^{-x^2}$,
\[
\begin{split}
|\tilde{g}(\xi) - \tilde{g}^N(\xi)| & \lesssim |\Sigma^{-1} (\xi - \Sigma y) - (\Sigma^N)^{-1}(\xi - \Sigma^N y)|  \\
& \leq | [ \Sigma^{-1} - (\Sigma^N)^{-1} ] (\xi - \Sigma y) | +  |\Sigma^N|^{-1} | (\Sigma - \Sigma^N) y |  \\
& \leq | (\Sigma^N)^{-1}( \Sigma - \Sigma^N ) \Sigma^{-1} ( \xi - \Sigma y) | +   |\Sigma^N|^{-1}|\Sigma - \Sigma^N| |y| \\
& \leq |(\Sigma^N)^{-1}| (|\Sigma^{-1}| 
%|\Sigma + \Sigma^N|  
+1)| \Sigma - \Sigma^N | %+   |\Sigma^N|^{-2} |\Sigma - \Sigma^N|
= \cO(N^{-\kappa/d}),
\end{split}
\]
where the last line follows from the boundedness of $\Sigma^{(N)}$, following from 
Assumption \ref{ass:dmfenkf}(ii) (see also the Remark \ref{rem:dmfenkfass}b).

The extension to non-diagonal $K^{(N)}$ follows from two arguments.  First, one has 
\begin{equation}
\label{eq:sigk}
| \Sigma - \Sigma^N | \lesssim | K - K^N|
\end{equation} 
 by continuity of the determinant,
hence the characteristic function $\rho(\lambda, A) = {\rm det} (\lambda I - A)$, as a function 
of the matrix $A$.   Define $B^{(N)} = \Sigma^{(N)} V^{(N)}$.
Now, similarly to above, we will need to bound 
$|A_{(\Gamma)}-A_{(\Gamma)}^{(N)}|$, 
where $A_{(\Gamma)}^{(N)} = B^{(N)} %\Sigma^{(N)} V^{(N)} 
(\Gamma) [B^{(N)}]^\top$,
%(V^{(N)})^\top \Sigma^{(N)}$,
and hence $|B^N-B|$.
%\Sigma^{N} V^{N} - \Sigma V|$. 
Now notice that
%\[
%\begin{split}
\begin{eqnarray}
\label{eq:a}
|A^N - A | & = & |[K^N]^\top K^N - K^\top K | \leq 
2 {\rm max}\{|K|,|K^N|\} |K-K^N|, \\
|[\Sigma^N]^2 - \Sigma^2 | & \leq & 2 {\rm max}\{|\Sigma|,|\Sigma^N| \} |\Sigma - \Sigma^N|,
\label{eq:sig}
\end{eqnarray}
%\end{split}
%\]
and
%\[
%\begin{split}
\begin{eqnarray}
\label{eq:agam}
|A_\Gamma^N - A_\Gamma | & = & |[K^N]^\top[\Sigma^N]^2 K^N - K^\top [\Sigma]^2 K | \\ 
& \leq  & 2 {\rm max}\{|K|,|K^N|\} |\Sigma^N| |K-K^N| + |K|^2 |[\Sigma]^2 - [\Sigma^N]^2 |. 
\nonumber
\end{eqnarray}
%\end{split}
%\]
Define $a^{(N)}= \xi - B^{(N)}y$.  Then, using continuity of $e^{-x}$, one must bound
%\[
%\begin{split}
\begin{eqnarray}
\nonumber
|\tilde{g}(\xi) - \tilde{g}^N(\xi)| & \lesssim& C | \langle a^N,A_\Gamma^Na^N\rangle - \langle a, A_\Gamma a \rangle| \\
%= C |\Gamma^{1/2} [B^N]^\top a^N| - B^\top \Gamma a \rangle|\\
 \nonumber
 & \leq & |\langle a, (A_\Gamma^N-A_\Gamma) a \rangle| + |\langle (a-a^N), A_\Gamma a\rangle| + |\langle a^N, A_\Gamma (a-a^N) \rangle| \\
& \leq & |a|^2 |A_\Gamma^N-A_\Gamma| + 2 {\rm max} \{|a|,|a^N|\}|A_\Gamma| |a - a^N|.
\label{eq:ga}
\end{eqnarray}
%\end{split}
%\]
Applying the bounds \eqref{eq:sig}, \eqref{eq:sigk}, and \eqref{eq:bndk} to \eqref{eq:agam}, and recalling the boundedness 
of the domain from Assumption \ref{ass:dmfenkf}(iii) takes care of the first term, and leaves $ |a - a^N|$
for the second term.
Now one must %use contradiction to 
show that $|V-V^N| \lesssim N^{-\kappa/d}$.
Choose a single eigenpair $(\lambda,v)$ associated to $A$, and one $(\lambda',v')$ associated to $A'$, where
$|\lambda - \lambda'| = \cO(\varepsilon)$ and $|A-A'| = \cO(\varepsilon)$.
%, but 
%$|v-v'|/\varepsilon \rightarrow \infty$.
One has the following
\[
\begin{split}
\frac{1}{\varepsilon} | v - v' | & \leq  (\varepsilon |A' -  \lambda' I | )^{-1} | (A' -  \lambda' I) (v-v') | \\
%= C |\Gamma^{1/2} [B^N]^\top a^N| - B^\top \Gamma a \rangle|\\
\nonumber
 & =  (\varepsilon |A' -  \lambda' I | )^{-1} |(A'-A)v + (\lambda - \lambda') v| \\
 &  \leq  (\varepsilon |A' -  \lambda' I | )^{-1} ( |A'-A| + |\lambda - \lambda'|) \leq C,
\end{split}
\]
for some $C$.  Hence $ | v - v' | = \cO(\varepsilon)$.  This proves $|V - V^N| = \cO(N^{-\kappa/d})$, 
following from \eqref{eq:a} and \eqref{eq:sig}.
Now observe
\[
\begin{split}
 |a - a^N| & = | ( \Sigma V - \Sigma^N V^N ) y | \leq | ( \Sigma - \Sigma^N) V y | + | \Sigma^N( V - V^N ) y | \\
  & \leq |V| | \Sigma - \Sigma^N | +  | \Sigma^N | |y|  | V - V^N |.
\end{split}
\]
Combining this with \eqref{eq:ga} then provides the desired point-wise bound 
\begin{equation}
|\tilde{g}(\xi) - \tilde{g}^N(\xi)| = \cO(N^{-\kappa/d}).
\label{eq:gfin}
\end{equation}

% \[
%\begin{split}
%|\tilde{g}(\xi) - \tilde{g}^N(\xi)| & \leq C |\Sigma^{-2} (\xi - \Sigma y) - (\Sigma^N)^{-2}(\xi - \Sigma^N y)|  \\
%& \leq | [ \Sigma^{-2} - (\Sigma^N)^{-2} ] (\xi - \Sigma y) | +  |\Sigma^N|^{-2} | (\Sigma - \Sigma^N) y |  \\
%& \leq | (\Sigma^N)^{-2}( \Sigma^{2} - (\Sigma^N)^{2} ) \Sigma^{-2} | |\xi - \Sigma y| +   |\Sigma^N|^{-2}|\Sigma - \Sigma^N| |y| \\
%& \leq |\Sigma^N|^{-2} (|\Sigma|^{-2} |\Sigma + \Sigma^N|  +1)| \Sigma - \Sigma^N | %+   |\Sigma^N|^{-2} |\Sigma - \Sigma^N|
%= \cO(N^{-\kappa/d}),
%\end{split}
%\]

Denote $p^{(N)}(\cdot) = {\rm det}[I-K^{(N)}_{t_j}H]^{-1}\tilde{\ml}^{(N)}_{t_{j}}((I-K^{(N)}_{t_j}H)^{-1}~\cdot~)$, and recall
equation \eqref{eq:preup} shows $d_\infty(p^{N}, p) = \cO( N^{-\kappa/d} )$.

Observe the decomposition
\begin{equation}
%\[
%\begin{split}
\cC^{G,N} p^N - \cC^G p   = (\cC^{G,N} - \cC^G) p + C^{G,N}(p^N - p). %\\
%& = \int \Big(\frac{\tilde{g}^N(\xi)}{\int \tilde{g}^N} - \frac{\tilde{g}(\xi)}{\int \tilde{g}}\Big)p (x - \xi) d\xi +
%\int \frac{\tilde{g}^N(\xi)}{\int \tilde{g}^N}  [p (x - \xi) - p^N (x - \xi) ] d\xi
%\end{split}
%\]
\label{eq:upsplit}
\end{equation}
%Notice that the second term is bounded
For the second term, Assumption \ref{ass:dmfenkf}(iii) provides
\begin{equation}
|\cC^{G,N}(p^N - p)(f)| = \Big | \int_\Omega f(x) \int_{H\Omega} \frac{\tilde{g}^N(\xi)}{\int_{H\Omega} \tilde{g}^N}  [p (x - \xi) - p^N (x - \xi) ] d\xi dx \Big | \leq C d_\infty(p^N,p),
\label{eq:cg1}\end{equation}
so \eqref{eq:preup} controls the second term.
The first term is decomposed as follows:
\[
\begin{split}
[(\cC^{G,N} - \cC^G) p ] (x) %\\
& = \int_{H\Omega} \frac{\tilde{g}^N(\xi)}{\int_{H\Omega} \tilde{g}^N \int \tilde{g}} p (x - \xi) d\xi \Big(\int_{H\Omega} [\tilde{g}(\xi) - \tilde{g}^N(\xi)] d\xi\Big) \\
& + \frac{1}{\int_{H\Omega} \tilde{g}} \int_{H\Omega} [\tilde{g}^N(\xi) - \tilde{g}(\xi)] p (x - \xi) d\xi.
\end{split}
\]
Assumption \eqref{ass:dmfenkf}(iii) and the pointwise bound \eqref{eq:gfin} finally allow one 
to conclude that 
\begin{equation}
d_\infty(\cC^{G,N}p, \cC^G p) = \cO(N^{-\kappa/d}).
\label{eq:cg}
\end{equation}
The induction is now complete so $d_\infty \left (\ml^N_{t_j},\ml_{t_j} \right) = 
d_\infty(\cC^{G,N} p^N,\cC^G p) = \cO\left(%\frac{1}
{N^{-\kappa/d}}\right)$, by \eqref{eq:upsplit}, \eqref{eq:cg1} and \eqref{eq:cg}.
%The proof follows 
%by first noticing that difference error in the finite approximation 
%of \eqref{eq:gupdate1} arises due to a quadrature rule for the integration 
%and the finite approximation of $\hat{g}_{t_j}$ as defined in Eq. \eqref{eq:gk}  
%by $\hat{g}_{t_j}^N$.  
%The approximation error arising from  $\hat{g}_{t_j}^N$ 
%inherits the error from the finite approximation of $K^N_{t_j}$, 
%hence $\widehat{C}^N_{t_j}$ from Section \ref{ssec:mfenkf} by continuity, 
%and hence finally the deterministic FP approximation of 
%$\cK_h \ml^N_{j-1}$, where $\cK_h$ is the forward kernel defined in 
%\eqref{eq:dtf2}, for time increment $h$.  
%The remainder of the proof follows 
%the one for particle filter convergence \cite{rebeschini2013can, lsz2015} 
%except with $\hat{g}_{t_j}$ replacing
%$g_{t_j}$.  In other words, the rate of convergence persists over multiple steps with 
%a constant growing no faster than geometrically in time.   
\end{proof}}

It is worthwhile to note that similar analysis using the 
full likelihood $g$ results in the same asymptotic estimate for the deterministic 
approximation of the full filtering distribution, the full FP filter.
As a consequence, the full FP filter
will asymptotically outperform the Monte Carlo particle filter, 
which has $\cO(N^{-1/2})$ rate of convergence to the actual filtering distribution, 
for $d<2\kappa$ in the case of this most naive discretization.

\subsection{Main theoretical result}
\label{ssec:maintheory}

%%%%%%%%%%%%%%%%%%%%%%%%%%%%%%%%%%%%%%%%%%%%%%%%%%%%%%%%%%%%%
\kl{Recall the definitions of the EnKF density $\hat{\ml}^N_{t_j}$,
the DMFEnKF density $\ml^N_{t_j}$, and the posterior $\rho_{t_j}$.}
It is clear that the error of the approximating density decomposes 
into (i) finite resolution error, arising in the form of either statistical error from the 
ensemble approximation, or discretization error for the FP filters, and (ii)
bias error, arising from the linear/Gaussian assumptions underpinning the method:
\begin{equation}
\arraycolsep=1.4pt\def\arraystretch{2.2}
\begin{array}{ccccc}
d\left (\hat{\ml}^N_{t_j},\rho_{t_j} \right ) &\leq& \underbrace{d \left (\hat{\ml}^N_{t_j},\ml_{t_j} \right )}_{\text{ensemble error}} & + &
\underbrace{d \left (\ml_{t_j},\rho_{t_j} \right )}_{\text{Gaussian error}}, \\
d\left (\ml^N_{t_j},\rho_{t_j} \right) &\leq& \underbrace{d \left (\ml^N_{t_j},\ml_{t_j} \right )}_{\text{discretization error}} & + &
\underbrace{d \left (\ml_{t_j},\rho_{t_j} \right)}_{\text{Gaussian error}}, 
\label{eq:er1step}
\end{array}
\end{equation}
{where one recalls the definition of the MFEnKF density $\pi_{t_j}$. }

Recall the operators $\cK_h$, 
$\cC_{t_j}$, and $\cC^G_{t_j}$ defined by 
Eqs. \eqref{eq:dtf2} (with subscript $h$ indicating the observation time increment), 
\eqref{eq:update}, and \eqref{eq:gupdate1},  
and the definitions of $g(\cdot,\cdot)$ and $\hat{g}_{t_j}(\cdot,\cdot)$
from Eqs. \eqref{eq:g} and \eqref{eq:gk}.
Note that, since $\cK_h$ is a Markov transition kernel, by Lemma 4.8 of \cite{lsz2015}
and Assumption \ref{ass:enkf}(i) (which allows generalization to the metric $d$ from $d_\infty$) 
\begin{equation}
d\left (\cK_h^\top \pl, \cK_h^\top \pl' \right ) \leq d\left( \pl, \pl' \right ), \quad \quad 
\forall ~ \pl, \pl'  \in \cP(\bbR^d), ~ \forall h>0. 
\label{eq:erfwd}
\end{equation} 
Define the Gaussian projection $G$ of a density $\pl\in \cP(\bbR^d)$ as follows
\begin{eqnarray}
m&=&\int u \pl(u) du, \\
C&=& \int \left [ (u-m) \otimes (u-m) \right] \pl(u) du, \\
%\hat{\pl}(u) &=& \exp \left \{ -\frac12\left | u - m \right|^2_C \right\}, \quad \forall u\in \bbR^d, \\
(G \pl)(u) & := &  (2 \pi)^{-d/2}|C|^{-1/2} \exp \left \{ -\frac12\left | u - m \right|^2_C \right\}, \quad \forall u\in \bbR^d.
%\frac{\hat{\pl}}{\int \hat{\pl}}.
\label{eq:gaussproj}
\end{eqnarray}
Now denote $\pl^G = G \pl$ and $\pl^{NG} = \pl - \pl^G$.  Notice that the remainder 
is not positive and $\int \pl^{NG} = 0$.

%Next we make 
Here an additional assumption is made of 
sufficient smoothness and decay of 
$\ml_{t_j}$ and $\rho_{t_j}$ so that 
\begin{equation}
\cK_h^\top \ml_{t_j} = \ml_{t_j} + \cO(h), \quad {\rm and} \quad
\cK_h^\top \rho_{t_j} = \rho_{t_j} + \cO(h),  
\label{eq:smooth_pert}
\end{equation}
where here and in what follows, $\cO$ will refer to boundedness up to a constant in the 
metric \eqref{eq:metric} except with the supremum taken over the class of functions
 $\{|f|_\infty \leq 1\} \supset \{|f|_L \leq 1\}$.

 \begin{theorem}
 \label{thetheo}
Given an observation {time} increment $h$, and under Assumptions \ref{ass:enkf} and \ref{ass:dmfenkf},
%\eqref{eq:dublin},  
the following asymptotic error bounds hold
as $h \rightarrow 0$ and $N\rightarrow \infty$
$$
d\left (\ml^N_{t_j},\rho_{t_j} \right) = \cO(N^{-\kappa/d}+\lambda h)  \quad \text{and} \quad 
d\left (\hat{\ml}^N_{t_j},\rho_{t_j} \right ) = \cO(N^{-1/2}+\lambda h)\quad ~ \text{for any finite} ~j,
$$
where $\lambda=0$ if $F$ in \eqref{eq:langevin} is linear and $\lambda=1$ if $F$ is nonlinear.
 \end{theorem}

 \begin{proof}
Recall Eqs. \eqref{eq:er1step}.  
For linear $F$  in \eqref{eq:langevin}, of course the Gaussian bias in \eqref{eq:er1step} vanishes, 
and the result is exactly given by Theorems \ref{enkftheo} and \ref{thm:dmfenkf}. %Eqs. \eqref{eq:erdisc}.
Consider the case of $F$ {\it nonlinear}.  
Since the first terms are bounded by Eqs. \eqref{eq:erdisc},
the focus here is on the second Gaussian error term, which is decomposed as follows
\begin{eqnarray}
\nonumber
d \left (\ml_{t_{j+1}},\rho_{t_{j+1}} \right) &=& d\left (\cC^G_{t_{j+1}} \cK^\top_h \ml_{t_j}, \cC_{t_{j+1}} \cK^\top_h \rho_{t_j} \right) \\
&\leq&  
d\left (\cC^G_{t_{j+1}} \cK^\top_h \ml_{t_j}, \cC_{t_{j+1}} \cK^\top_h \ml_{t_j} \right) + d\left (\cC_{t_{j+1}} \cK^\top_h \ml_{t_j}, \cC_{t_{j+1}} \cK^\top_h \rho_{t_j} \right).
\label{eq:pfdecom}
\end{eqnarray}

Let $\pl$ be a probability density and assume that 
\begin{equation}
\left | \int f \pl^{NG}du \right | \leq c h < \alpha^2/6
\label{eq:ass1}
\end{equation}
 almost surely
for all $|f|_\infty \leq1$.  
Then for all $|f|_\infty \leq 1$
\begin{eqnarray}
\label{eq:big1}
\int f(u) (\cC_{j} \pl)(u)du & = & \frac{\int g_{j}f \pl^G } {\int g_{j} \pl^G + \int g_{j}\pl^{NG}}  +  
\frac{\int g_{j}f \pl^{NG}} {\int  g_{j} \pl} \\
\label{eq:big2}
&=& \frac{\int f g_{j} \pl^G} {\int g_{j}\pl^G} \sum_{k=0}^\infty 
\left [ -\frac{1}{\alpha} \frac{\int \alpha g_{j}\pl^{NG}} {\int g_{j}\pl^G} \right]^k + 
\alpha^{-1} \frac{\int \alpha g_{j}f \pl^{NG}} {\int g_{j} \pl} \\
\label{eq:big3}
&\leq& \frac{\int f g_{j} \pl^G} {\int g_{j} \pl^G}  
+ 2ch\alpha^{-2}
+ \alpha^{-1} \left | \frac{\int \alpha g_{j}f \pl^{NG}} {\int g_{j} \pl}\right | \\
&\leq&  \int f(u) \left( \cC_{t_j}^G \pl^G \right )(u) du  + 3ch \alpha^{-2}.
\label{eq:cgauss1}
\end{eqnarray}
To see the step between \eqref{eq:big1} and \eqref{eq:big2} notice first that by assumption 
\eqref{eq:dublin} $|\alpha g_{t_j}|_\infty \leq 1$ and inf$(g_{t_j}) \geq \alpha$.  Also, since $p^G$ is 
a positive probability density, the denominator of the bracketed expression is bounded below by $\alpha^2$, 
and by assumption $\epsilon := {\rm sup}_{|f|_\infty \leq 1} | \int f \pl^{NG}du | \alpha^{-2} \leq c h \alpha^{-2} < 1/6$.
Between \eqref{eq:big2} and \eqref{eq:big3} we have used that $g_{t_j}p^G/\int{g_{t_j}p^G}$
is a probability measure and $|f|_\infty \leq 1$, as well as the fact that $\epsilon/(1-\epsilon)<2\epsilon$
for $0<\epsilon<1/2$.
By a similar argument,
\begin{eqnarray}
- \int f(u) (\cC_{t_j}^G \pl)(u) du &\leq& -\int f(u) \left( \cC_{t_j}^G \pl^G \right )(u) du +  3ch \alpha^{-2}.  
\label{eq:cgauss2a}
\end{eqnarray}
Evaluating the same two inequalities with the signs swapped then gives  
\begin{eqnarray}
\left | \int f(u) (\cC_{t_j} \pl)(u) du - \int f(u) (\cC_{t_j}^G \pl)(u) du \right | &\leq& 
6 c h \alpha^{-2}.
\label{eq:cgauss2}
\end{eqnarray}

Given assumption \eqref{eq:smooth_pert}, and 
under the a priori assumption that $\ml^{NG}_{t_j} = \cO(h)$ and $\rho^{NG}_{t_j}=\cO(h)$, we have 
\begin{eqnarray} 
( \cK^\top_h \rho_{t_j} )^{NG} =  \cO(h), \quad \quad ( \cK^\top_h \ml_{t_j} )^{NG} =  \cO(h).
\label{eq:oh}
\end{eqnarray}
Therefore, 
\begin{eqnarray} 
d_\infty \left (\cC^G_{t_j} \cK^\top_h \ml_{t_j}, \cC_{t_j} \cK^\top_h \ml_{t_j} \right) & \leq &  6 c h \alpha^{-2},
\label{eq:pft1}
\end{eqnarray}
where $d_\infty$ denotes the metric \eqref{eq:metric} except with the supremum taken over 
$\{|f|_\infty\leq 1\}$.  The same reasoning shows that  ${\rho}^{NG}_{t_{j+1}} =  \cO(h)$ and 
${\ml}^{NG}_{t_{j+1}} =  \cO(h)$, using the second assumption of \eqref{eq:dublin} for the latter,
hence justifying the a priori assumption above \eqref{eq:oh} by induction.  
The first term of \eqref{eq:pfdecom} is bounded by \eqref{eq:pft1} since 
$\{|f|_\infty \leq 1\} \supset \{|f|_L \leq 1\} \Rightarrow d(\cdot,\cdot) \leq d_\infty(\cdot,\cdot)$.

For the second term of  \eqref{eq:pfdecom} it is well-known 
(see \cite{rebeschini2013can}, and also Lemma 4.9 of \cite{lsz2015})
that
\begin{eqnarray} 
d\left (\cC_{t_j} \cK^\top_h \ml_{t_j}, \cC_{t_j} \cK^\top_h \rho_{t_j} \right) & \leq &  2\alpha^{-2} d({\ml}_{j},{\rho}_{j}).
\label{eq:pft2}
\end{eqnarray}
Inserting Eqs. \eqref{eq:pft1} and \eqref{eq:pft2} into \eqref{eq:pfdecom} yields
$$
d \left (\ml_{t_{j+1}},\rho_{t_{j+1}} \right) \leq 2\alpha^{-2} d({\ml}_{j},{\rho}_{j}) + 6 c(j) h \alpha^{-2} ,
$$
Noting that $\ml_0 = \rho_0$, 
and combining the above with Eqs. \eqref{eq:erdisc} and \eqref{eq:er1step} gives, 
for some $C(j)>0$
\begin{eqnarray}
\nonumber
d\left (\hat{\ml}^N_{t_j},\rho_{t_j} \right ) &\leq& C(j) (h + N^{-1/2}), \\
d\left (\ml^N_{t_j},\rho_{t_j} \right) &\leq& C(j) (h + N^{-\kappa/d}). 
\label{eq:pfdone}
\end{eqnarray}

\end{proof}

Note that the bound given above grows exponentially with the number of observation times.
Under the assumption of a Dobrushin ergodic coefficient \cite{dobrushin1956central}, i.e. the dynamics of $\cK_h$
"mix sufficiently well", it may be possible to obtain a uniform bound on the constant, as has been 
done for standard particle filters in \cite{del2001stability, rebeschini2013can}.
\kl{Also note the assumption \eqref{eq:ass1} 
may require small $h$ 
because the required constant may grow with each step.  
This would limit the results to small total time $t_j = j h$, 
which would not be very valuable.  There are two things to note here.  The first is that the error 
decays after a certain number of steps due merely to decorrelation, assuming some sufficient regularity of 
the forward kernel, merely due to ergodicity.  See \cite{mattingly2002ergodicity} for some general conditions for ergodicity of models
given by SDE (even with degenerate noise) and \cite{del2001stability, rebeschini2013can} 
for the derivation of time-uniform estimates for particle filters under such ergodicity assumption. 
The infinite-time horizon is not considered here.
The second is that the error which is retained over multiple time-steps is related to 
{\it non-Gaussianity} \eqref{eq:oh}. 
This may be small for two reasons, in addition to shortness of assimilation window $h$:
(i) tightness of the density due to accurate observations, 
following from Laplace approximation arguments \cite{laplace1986memoir},
and (ii) approximate linearity of the forward model. 
These aspects do not appear in the continuity assumption of \eqref{eq:smooth_pert},
and indeed they are hidden in the constants of \eqref{eq:ass1} and \eqref{eq:oh}.
In fact, on a single assimilation window, a shorter time interval $h$ allows a more accurate linearization in time,
and a more accurate observation tightens the distribution along a given subspace and allows a 
more accurate linearization in space.  
It is of interest to rigorously incorporate these subtleties into the theory
and to extend the current results to long time-horizons.  This is the subject of ongoing work.
Nonetheless, it will be shown numerically in the following section %in the numerical results 
that the results indeed extend to long times.
Note if there is a decoupled and strongly nonlinear 
direction of the dynamics that is left completely unobserved
then the small bias result is expected to 
hold only for small time.}  
%Indeed this would result in any filter
%being unstable, including the filtering distribution itself (in an appropriate sense).}

It is also worth mentioning that, although the measures are all random in the observation sequence
$Y_{t_j}$, this randomness does not appear explicitly in the proofs (it would in principle appear in \eqref{eq:pft1}, 
but its presence is innocuous and does not spoil the estimates).  
The only randomness which appears in the proofs
is the randomness of the ensemble in the EnKF distribution $\hat{\ml}^N_{t_j}$.  
The results therefore hold {uniformly} 
with respect to the observation sequence, and
one may consider the expectation only with respect to the ensemble.

\section{Numerical examples}
\label{sec:numerics}

This section will be divided into the following subsections. 
In Section
\ref{ssec:numintro}, the particular numerical set-up is described.
Results for a linear example follow in Section \ref{ssec:linear}, and results for 
a nonlinear example are presented in Section \ref{ssec:nonlinear}.

\subsection{Particular set-up}
\label{ssec:numintro}

Referring back to Sec. \ref{sec:fpf}, the following special case of Eq. \eqref{eq:langevin}
will be considered.  Let $d=1$ and let $$V(u)= - \int_{-\infty}^u F(s) ds$$ and require
$V \rightarrow \infty$ as $u\rightarrow \pm \infty$.  
In the absence of observations, we know that 
$$u(t) \rightarrow \frac{\exp(-V/b)}{\int \exp(-V/b)}$$ in distribution 
for any initial condition $u$.
In particular, we will consider the linear case $F(u)=-au$, corresponding to a single-well potential, in 
Section \ref{ssec:linear}, and the nonlinear case 
$F(u) = a u(1-u^2)/(1+u^2)$, {corresponding to a double-well potential}, in Section \ref{ssec:nonlinear}, for $a>0$.

This model is ubiquitous in the sciences, notably in molecular dynamics where it represents the motion of a particle with negligible mass trapped in an energy potential $V$ with thermal fluctuations represented by the Brownian motion \cite{Ris84}.

\kl{A second-order finite difference discretization is employed with exact semigroup integration in time, and 
trapezoidal quadrature rule.  
Further details and discussion of the particulars of the numerical discretization and error metric 
used in this section are \kl{included} in the supplementary materials.
The outer framework of filter convergence is the objective of the present work.
The point is that this holds for {\it any convergent numerical discretization}
at the inner level, and the order of approximation error is preserved.}

\todo{All examples here have $d=1$ and $\kappa=2$.  In this regime, Thoerem \ref{thetheo} indicates
that DMFEnKF will converge faster than EnKF, up to the $\cO(h)$ Gaussian bias.  The numerical results
will verify this.}

\subsection{Linear example}
\label{ssec:linear}

Here the 
Ornstein-Uhlenbeck process is considered in which $$F(u)=-au$$
in Eq. \eqref{eq:langevin}.
For observation increment $h$, this fits into the framework of
Equation \eqref{eq:dtf1}, and we have 
\begin{equation}
\Psi(u_{t_j}) = e^{-ah}u_{t_j},
\label{eq:ou}
\end{equation}
and $\Sigma = (b/a) (1 - e^{-2ah})$. For this example,
we let $H=\Gamma=a=b=h=1$.  
The potential $V$ is quadratic in this case
so the unconditioned process has uni-modal (Gaussian) invariant distribution.
In particular, in the absence of observations
$u(t) \rightarrow N(0,1)$ for any initial condition $u_0$.

See Fig. 5.1 for the error of the
deterministic MFEnKF-G2 Fokker-Planck filter (same in this case as G1) 
and the EnKF with respect to the mean and covariance of the true 
Kalman filter solution over several different values of $N$.
Note the error of EnKF is $\cO(1/\sqrt{N})$ 
while the DMFEnKF has error $\cO(N^{-\kappa})$ as proven in 
Sec. \ref{ssec:theo}. 
In fact, the error of the latter 
reaches numerical precision for 
$N=200$ \kl{degrees of freedom}, which is unreasonably good.  
This may be attributed to the fact that the numerical method preserves
the invariant distribution of the hidden process, 
and the observation increment is of the same order
as the relaxation time of the hidden process.
Notice EnKF with $N=4\times10^5$ 
still performs worse than MFEnKF 
with $N=40$.

\begin{figure}[!t]
\label{f2} 
\includegraphics[width=.45\columnwidth]{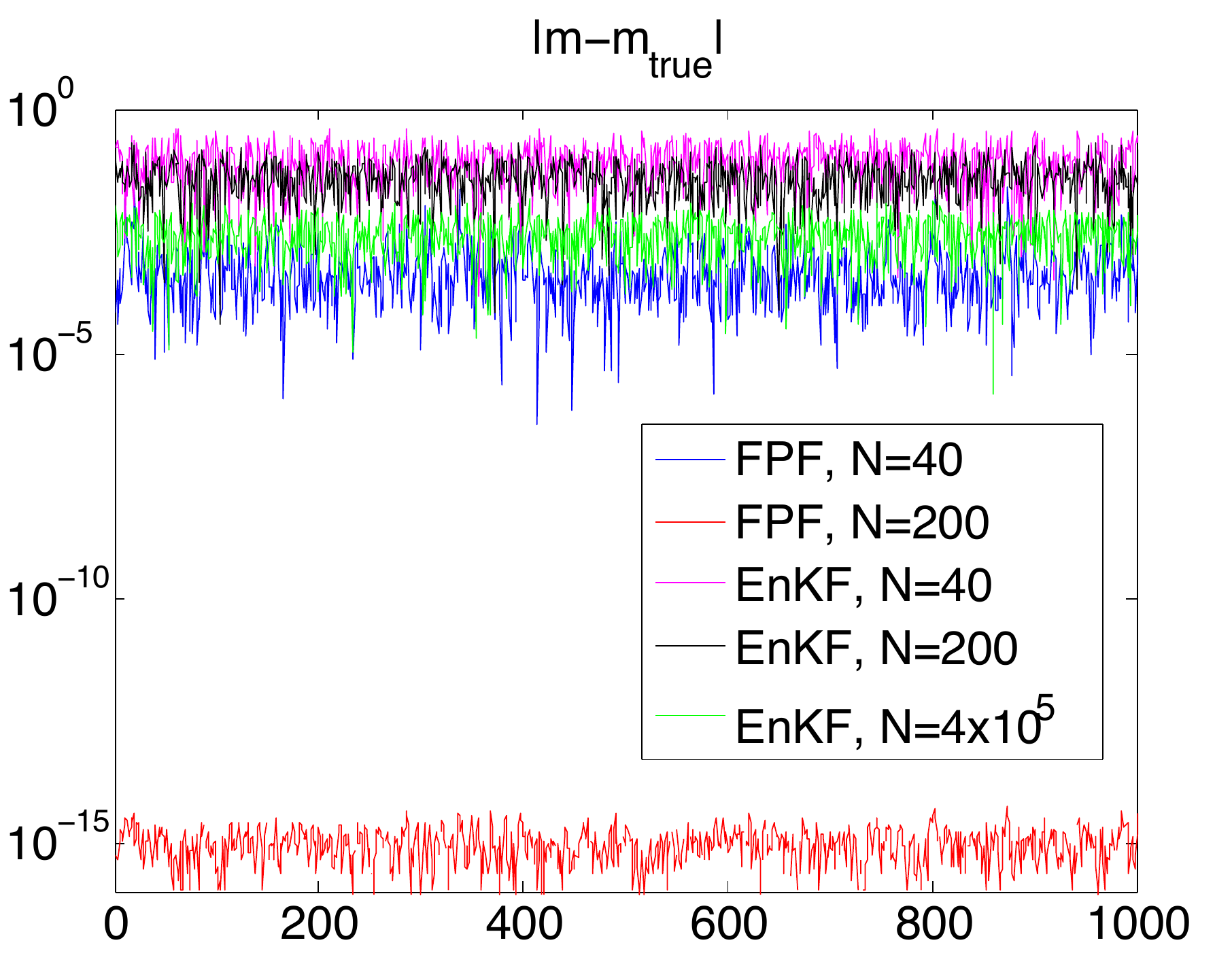}
\includegraphics[width=.45\columnwidth]{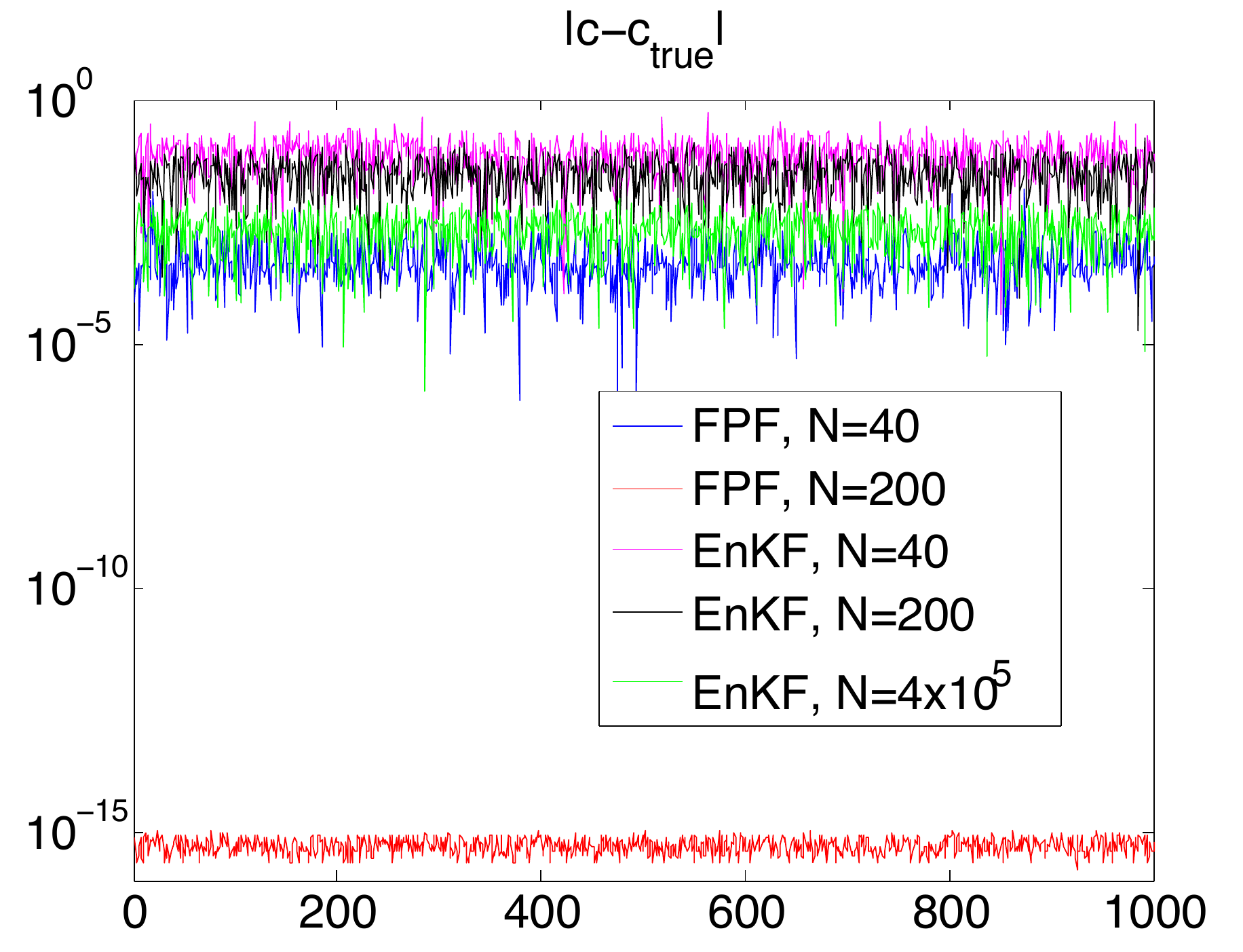}
\caption{OU process example.
The relative error in the mean (left), and covariance
  (right) with respect to the true (Gaussian) filtering distribution
for the Fokker-Planck filter MFEnKF-G2 (FPF) 
and the EnKF for several different cost levels $N$.  
Note the difference for EnKF is small owing to the 
$O(1/\sqrt{N})$ error, 
while the MFEnKF-G2 
reaches numerical precision for 
$N=200$. EnKF for $N=4\times10^5$ still performs
worse than MFEnKF-G2 with $N=40$.}
\vspace{3mm}
\end{figure}

\subsection{Nonlinear example}
\label{ssec:nonlinear}

\begin{figure}[!t]
\label{flang} 
\includegraphics[width=.45\columnwidth]{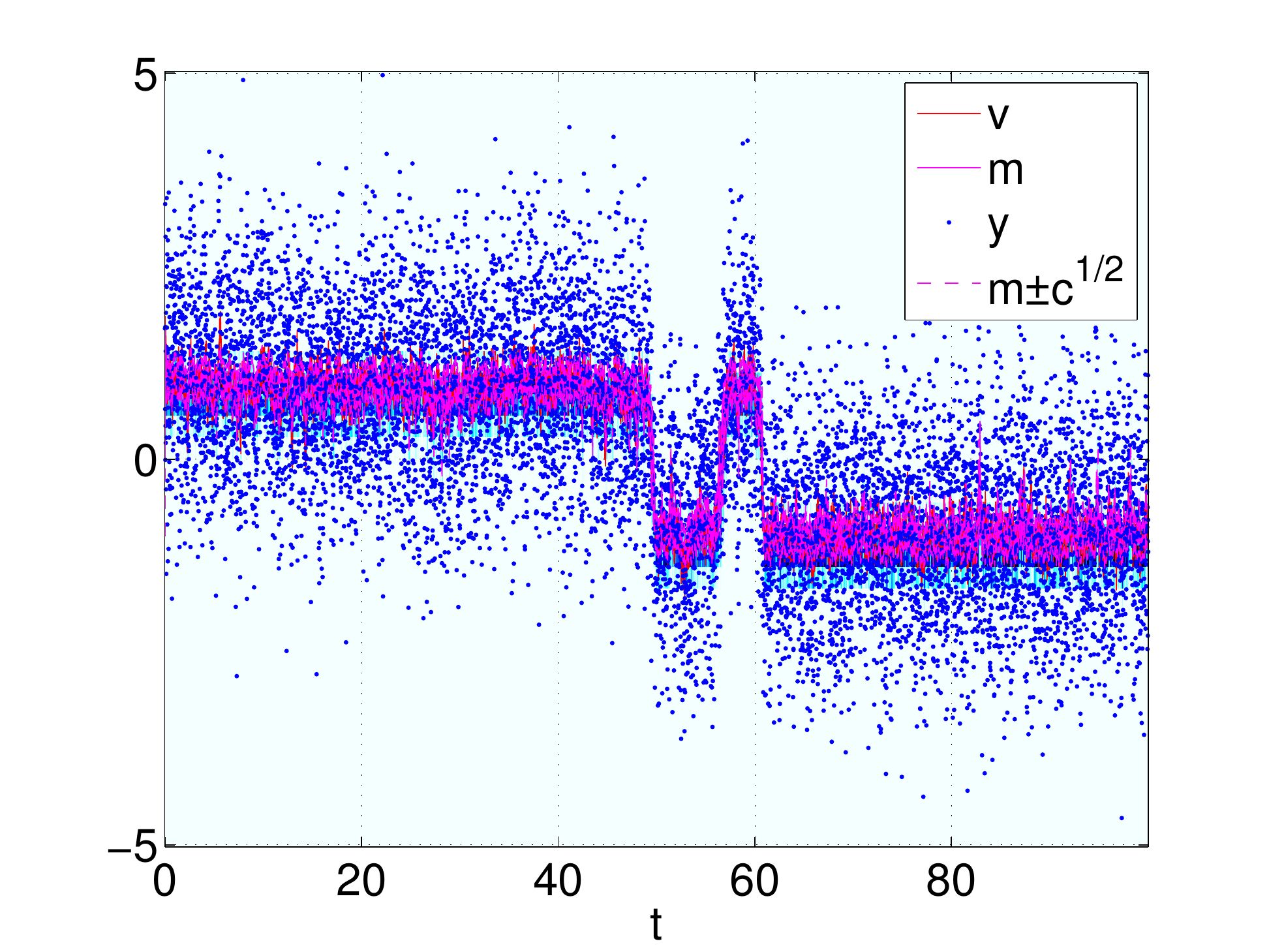}
\includegraphics[width=.45\columnwidth]{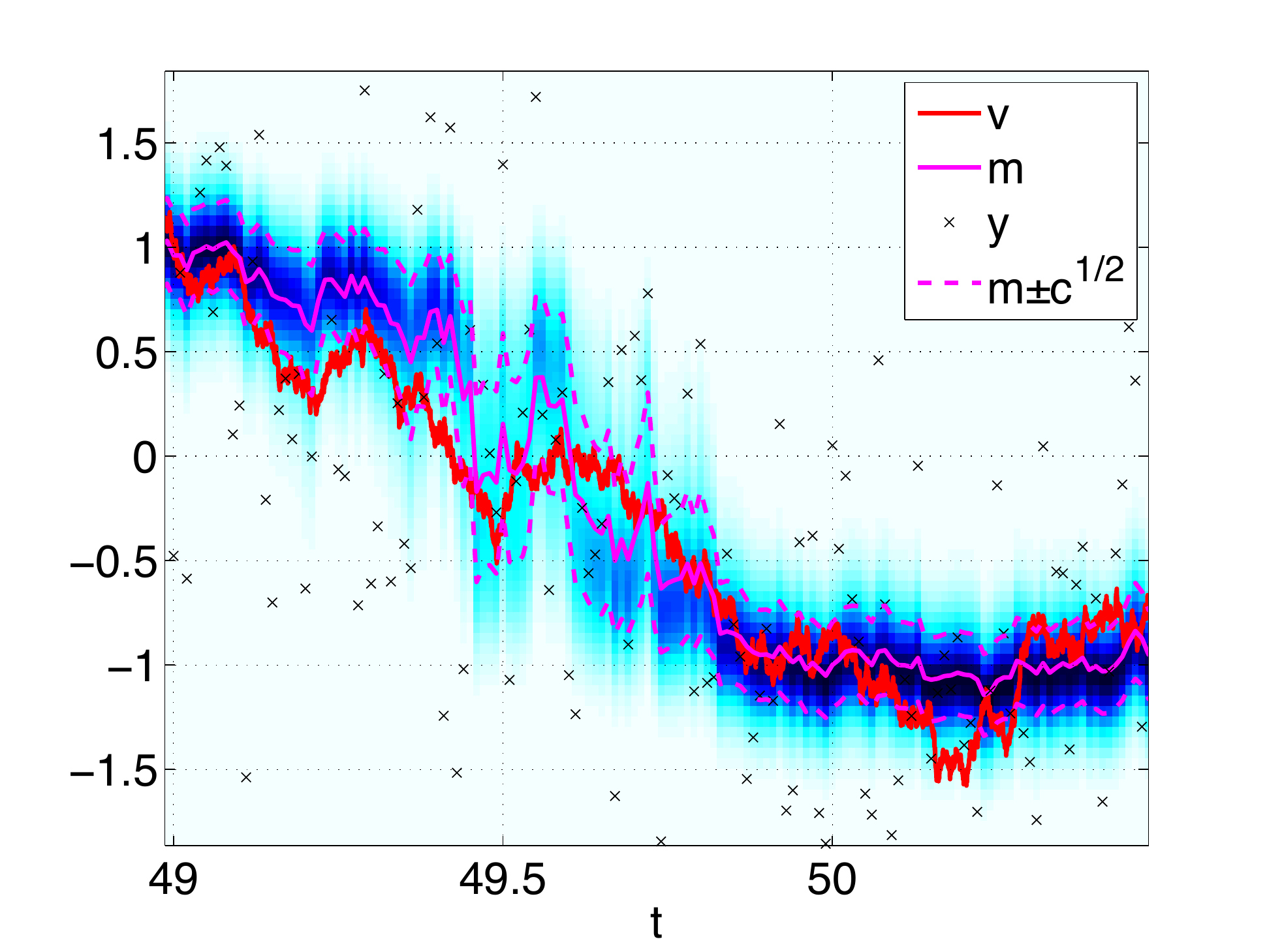}
\caption{The evolution (left) and close-up (right) of the filtering distribution
  for the continuous time Langevin process with a double-well external potential
  with discrete observations ($h=10^{-2}, dt=10^{-4}, \gamma=1$).  
  \kl{The blue background shows the probability density, % intensity, 
  with darker corresponding to higher density. %intensity.
  The symbols $v$, $m$, $y$, and $c$ correspond to the signal (or "truth"), the mean, 
  the observation, and the variance.}}
\vspace{3mm}
\end{figure}

The Langevin diffusion process is considered here, in which $F(u) = a u(1-u^2)/(1+u^2)$
in Eq. \eqref{eq:langevin} with $a=10$.  The dynamical noise level is $b=1/2$ and the 
observational noise variance is $\gamma^2=1$.
The invariant distribution of the unconditioned state is again 
known, as described in the introduction to the section.  
Since $V$ has a double-well shape in this case, the invariant distribution is bimodal 
and paths of Equation \eqref{eq:langevin} transition from one well to the other with a certain temporal probability. 
There is no analytical solution for the transient
state, as there was in Section \ref{ssec:linear}, and hence 
the transition kernel cannot be expressed in closed
form.

Upon Euler-Marayuma discretization of time-step $dt$, we have
\begin{equation}
u_{k+1} = u_k + dt f(u_k) + \sqrt{2 b dt} \Delta_k,
\label{euler_marauyma}
\end{equation}
where $\Delta_k \sim N(0,1)$ are i.i.d.  The discrete process $\{u_k\}$ does 
not have the same invariant distribution as $u(t)$ unless this is enforced
with an accept-reject step \cite{bou2010pathwise}.
Figure \ref{flang} shows the evolution of the filtering density in the background, 
with the \kl{truth ($v$), mean ($m$), standard deviation intervals ($m\pm c^{1/2}$), 
and observations ($y$) for a long time interval on the left, and a close-up interval on the right.}
In the case that the observation increment $h=ndt$ for $n > 1$, we solve
\eqref{euler_marauyma} for $n$ steps, and this approximates a draw from
the kernel appearing in \eqref{eq:dtf2}. 
Notice that this is a nonlinear {\it and non-Gaussian} state-space model, 
and we are now in the generalized framework.

A systematic series of numerical experiments is now performed with $dt=10^{-4}$
and $h=ndt$ where $n=5,20,100,$ and $1000$.  The solution 
from algorithm full FPF with $N=1000$ is taken as the benchmark against which to evaluate 
the other methods, and \kl{we look at the relative RMSE over the time window $t\in[0,1000]$ 
with respect to the truth, and with respect to the mean and variance. 
The time window averaged over includes hundreds
of transitions between wells, and so the results sufficiently incorporate transition behavior.}
The error of the full FPF with $N=200$ may be taken as the lower-bound error
level of the benchmark.

The results are graphically presented in Figs. \kl{6.3, 6.4, and 6.5. 
The term "truth" in what follows refers to the signal realization $u$ which gave rise to the data, i.e. the {\it realization} 
$u$ which, together with the realization $\eta$, gave rise to the observed sequence $y$ with $y_j = H u_j + \eta_j$.}
The following points summarize the plots:
\begin{itemize}
\item In the case of $n=5$ time-steps per observation, MFEnKF fails with $N=200$ degrees of freedom 
and requires $N=1000$ for convergence  due to error in computation of the convolution 
\footnote{In fact the convolution in MFEnKF is 
computed using FFT for convenience, resulting in a Riemann sum approximation 
of the integrals and leading to $\kappa=1$ rather than $\kappa=2$ here.}.  
This is presumably because of the reduced relative resolution for the more narrow distribution in this case.  
For the other Fokker-Planck based algorithms the distribution is imperceptibly close for $N=200$ and $N=1000$ -- this is a common criterion for determining convergence of numerical discretization.
\item Only in the strongly nonlinear and non-Gaussian case, for $n=1000$, is there a notable difference between the methods 
in relative RMSE with respect to the truth.  This is due presumably to the fact that in the other cases the RMSE between the truth and the mean of the filtering distribution is larger than the error between the mean and any of the approximate estimators.  In this case,
we observe the following from the close-up panel:
\begin{itemize}
\item The mean of the full FPF gives the minimum RMSE.  
\item The second place goes to MFEnKF-G1, which imposes Gaussianity only after the update.  
It is notable that in this strongly nonlinear and non-Gaussian case 
the RMSE is actually {\it smaller} if one performs the full nonlinear update and then imposes Gaussianity, rather than performing the linear  
update but retaining non-Gaussianity as in the MFEnKF.  
\item The RMSE of MFEnKF-G2, which imposes Gaussianity of the forecast, 
is by far the greatest.  The update to the mean and covariance here is the same as with the MFEnKF, but the resulting distribution does not retain any non-Gaussianity.  It is therefore natural to expect worse performance.
\item The RMSE of EnKF is approximately equal to MFEnKF.  One may therefore conclude that the linear error term is dominating.  The RMSE of EnKF with a 200 member ensemble is close to but slightly greater than that with a 1000 member ensemble, and the latter is closer to MFEnKF as expected.  
\end{itemize}
\item For smaller numbers of time steps between observations, and hence a closer to linear and Gaussian kernel, the improvement of the MFEnK filters over the traditional EnKF with respect to mean and variance increases until for $n=5$ the RMSE of EnKF is almost an order of magnitude greater than that of the MFEnK filters.
\item The statistics of the various MFEnK filters also converge to each other as $n$ decreases., indicating convergence to Gaussianity.
\item The statistics of the Full FPF with 200 degrees of freedom indicates an accuracy level of the benchmark results, and this level also coincides with the MFEnK filters for $n=5$. 
\item The RMSE with respect to both mean and variance of MFEnKF-G1 is noticeably smaller than the other estimators for $n=1000$.
\item For $n=5$ the RMSE of EnKF with $N=200$ is approximately twice that of EnKF with $N=1000$.  This difference decreases as non-Gaussian error begins to play a role, until for $n=1000$ there is very little difference. 
\item \kl{With respect to tracking well-transitions in the strongly non-Gaussian regime ($h=1000dt=0.1$), 
it is worth emphasizing the observation that it is preferable to 
preserve the non-linearity of the update and {\it impose} Gaussianity, than to preserve some amount of non-Gaussianity 
and perform a linear update as in MFEnKF (see the $\square$ symbols associated to MFEnKF-G1).  Of course
imposing Gaussianity before the update results in a linear update anyway and performs the worst 
(see the $\times$ symbols associated to MFEnKF-G2).  
This can be related to the idea of implicit filtering \cite{chorin2010implicit, vanden2012data} 
and fixed-lag smoothing, and can build upon the observations related to tracking transitions from
\cite{mandel2009ensemble}.  In the latter work, it was observed that a standard particle filter may fail 
to track due to lack of spread.  This property is not shared by the deterministic approximation.
In this regime the ranking of the methods in terms of signal tracking and 
in terms of approximation of the filtering distribution exactly coincide.}
\end{itemize}

\begin{figure}[!t]
\label{f5} 
\includegraphics[width=1\columnwidth]{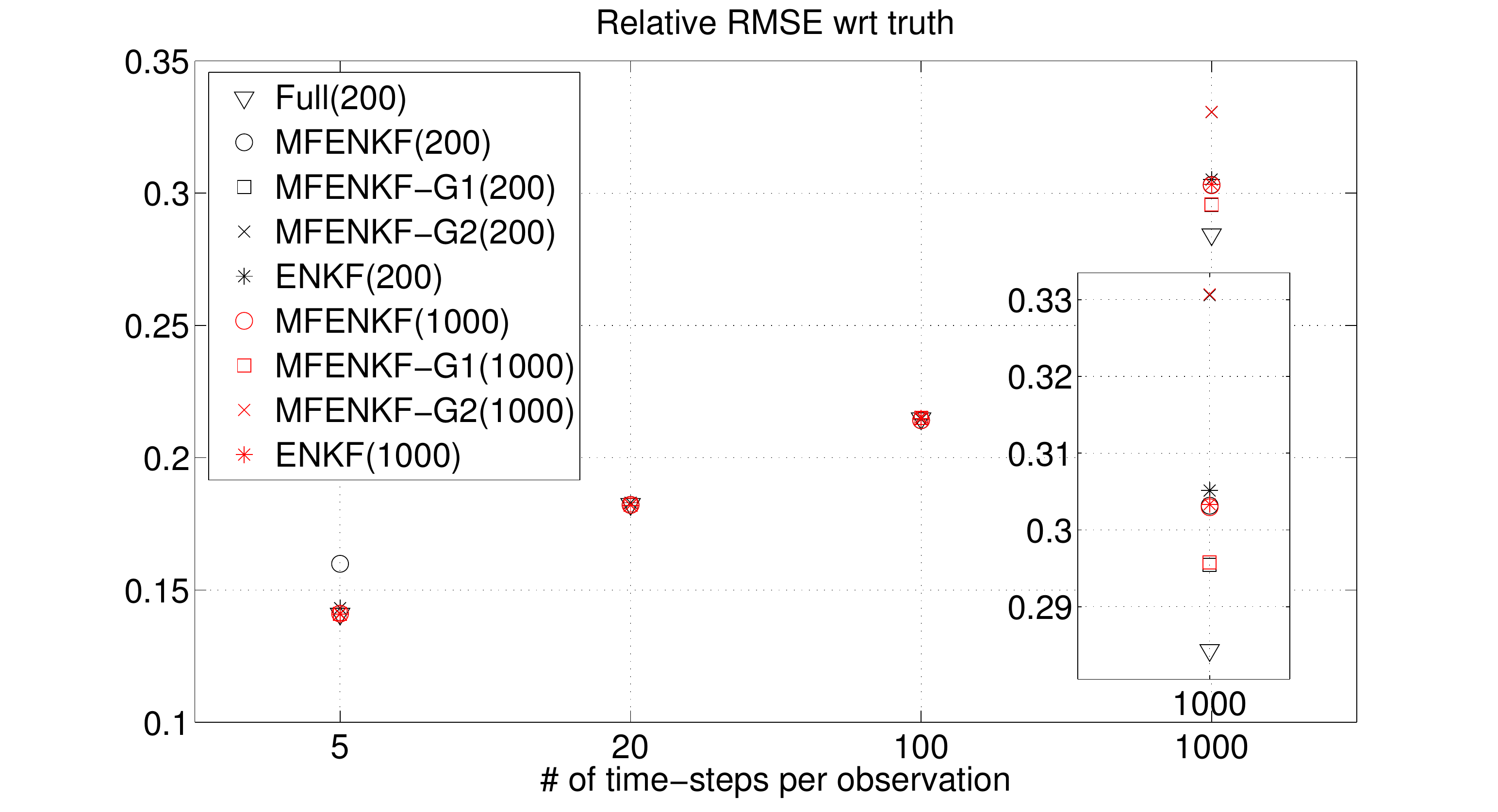} 
\caption{Relative RMSE with respect to the truth of the various filters over
various number of time-steps between observation updates 
with $dt=10^{-4}$.}
\vspace{3mm}
\end{figure}

\begin{figure}[!t]
\label{f6} 
\includegraphics[width=1\columnwidth]{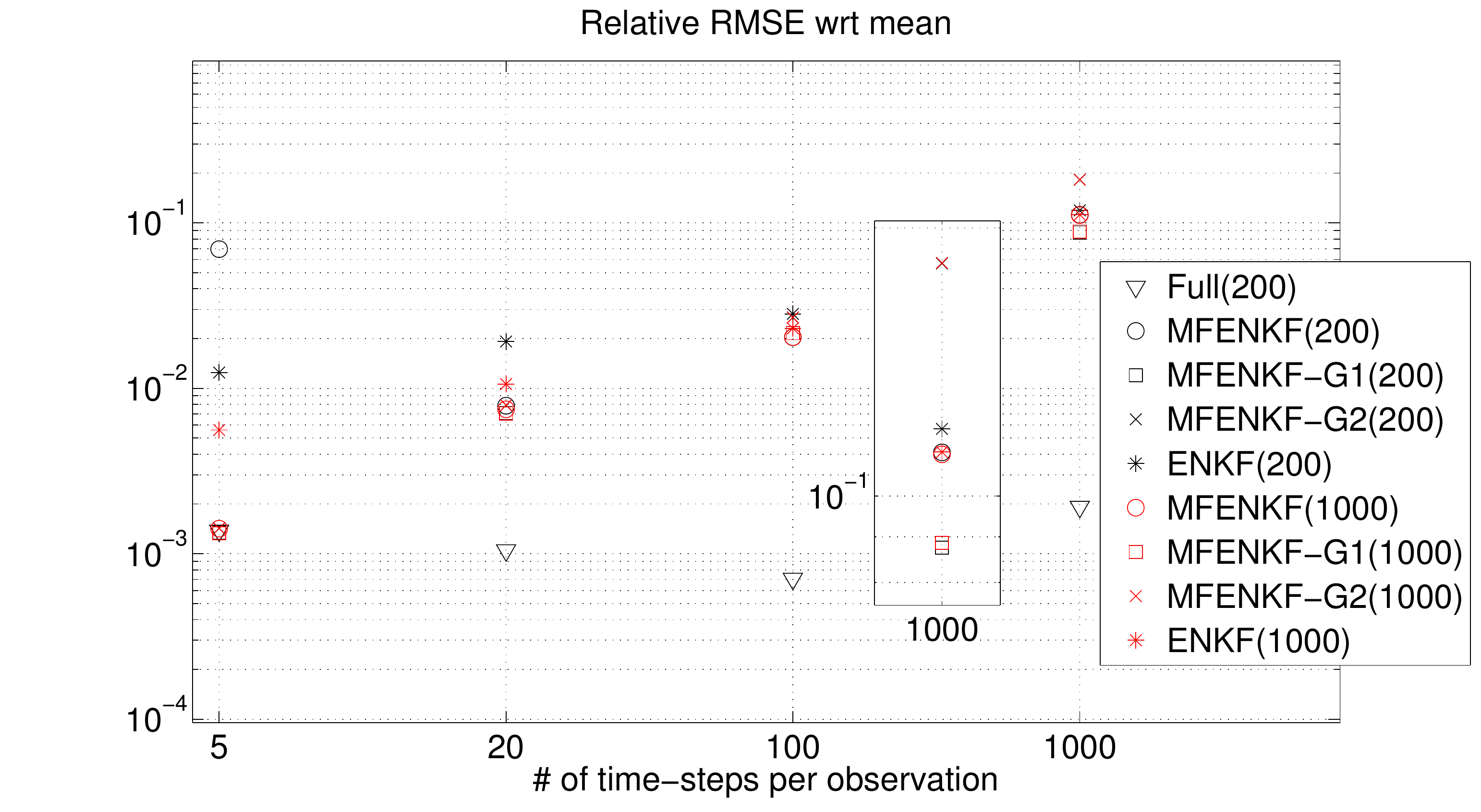} 
\caption{Relative RMSE with respect to the mean of the various filters over
various number of time-steps between observation updates 
with $dt=10^{-4}$.}
\vspace{3mm}
\end{figure}

\begin{figure}[!t]
\label{f7} 
\includegraphics[width=1\columnwidth]{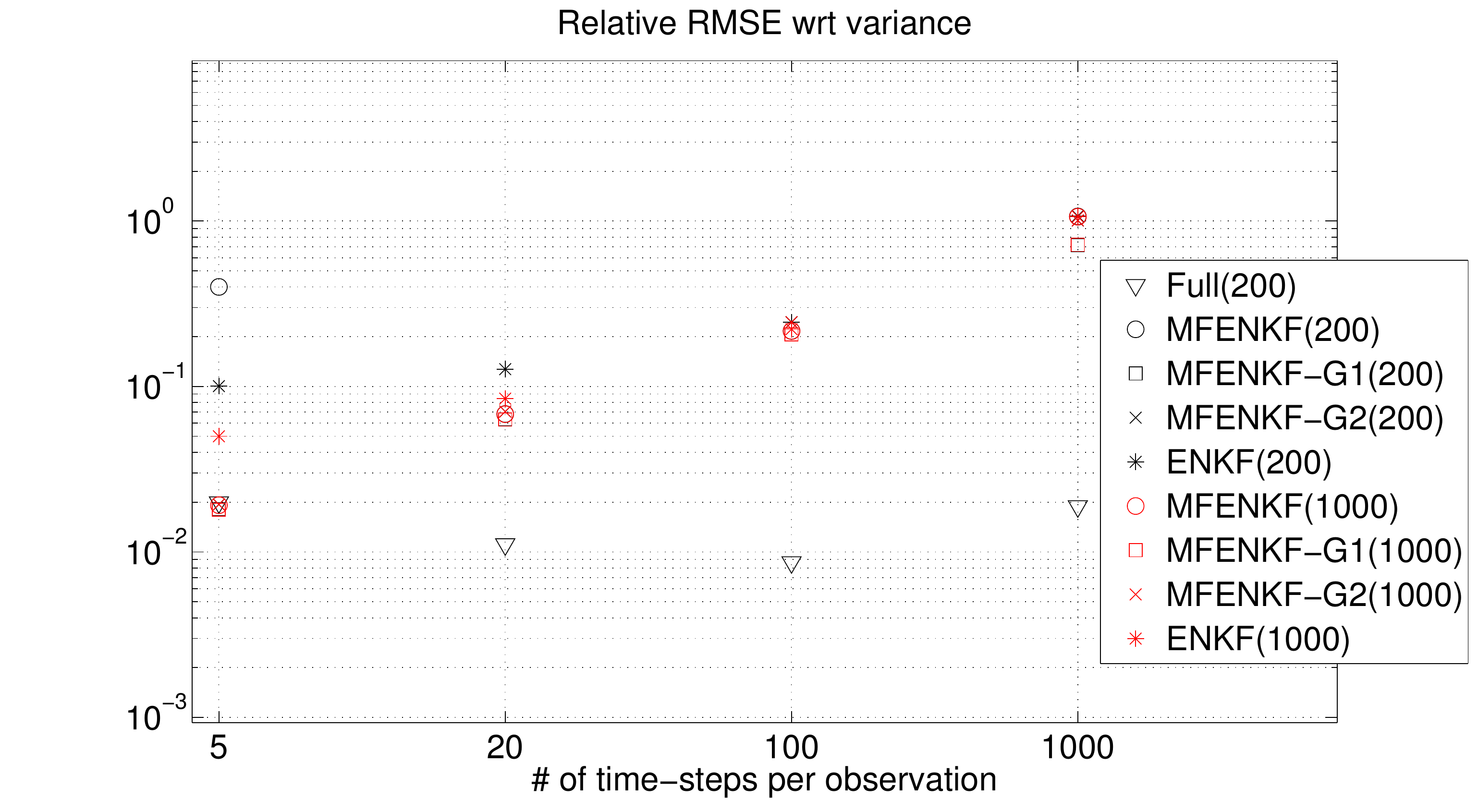} 
\caption{Relative RMSE with respect to the variance of the various filters over
various number of time-steps between observation updates 
with $dt=10^{-4}$.}
\vspace{3mm}
\end{figure}

\section{Conclusion}
\label{sec:conc}

A deterministic approach to 
mean-field ensemble Kalman filtering is considered for 
discrete-time observations of
continuous stochastic processes.  
This approach is based on deterministic approximation of the 
density by numerical solution of the Fokker-Planck equation, and 
intermittent updates to the density based on quadrature rules.
The scheme has a better rate of convergence in terms of degrees of freedom
than corresponding Monte Carlo based methods for $d < 2\kappa$, where
$\kappa$ is the minimal order of the PDE solve and quadrature rule for $d=1$.  
In particular, the proposed DMFEnKF scheme converges faster to its
linearly-biased limiting distribution than traditional EnKF methods, 
and the corresponding full filtering scheme converges faster to the true 
distribution than particle filtering methods.
These results can be used to develop more effective filters.
The use of more sophisticated discretization schemes should enable
the extension of these methods to a much  
higher dimension $d$.
Furthermore, it is well-known that very high-dimensional models may exhibit 
nonlinearity/instability/non-Gaussianity only on low-dimensional manifolds, so it is 
conceivable that deterministic solution techniques such as those developed here 
can be used on such manifold and combined with less expensive 
approximations on the complement.  
Hence it may be possible to extend such methods 
to real-world applications in the foreseeable future.
%
%A sequel to this work is in development, in which the deterministic 
%full Fokker-Planck filter is investigated as an alternative to particle 
%filtering methods.

%\newline
\noindent{\bf Acknowledgements}
{The authors thank Jan Mandel for invaluable feedback in the initial revision, as well as the two referees 
%(one of whom was Jan Mandel) 
whose careful and thorough review and suggestions have enormously improved the manuscript.}
Research reported in this publication was supported by the King Abdullah University of Science and Technology (KAUST).
R. Tempone is a member of the KAUST SRI Center for Uncertainty Quantification. 
K.J.H.Law is a member of the Computer Science and Mathematics Division at Oak Ridge National Laboratory and 
was supported in part by KAUST SRI-UQ and in part by an ORNL LDRD Strategic Hire grant.

%% BIBLIOGRAPHY
\bibliographystyle{siam}
\bibliography{mybib}

\def\cprime{$'$} \def\cprime{$'$} \def\cprime{$'$} \def\cprime{$'$}
  \def\cprime{$'$} \def\cprime{$'$} \def\cprime{$'$}
  \def\Rom#1{\uppercase\expandafter{\romannumeral #1}}\def\u#1{{\accent"15
  #1}}\def\Rom#1{\uppercase\expandafter{\romannumeral #1}}\def\u#1{{\accent"15
  #1}}\def\cprime{$'$} \def\cprime{$'$} \def\cprime{$'$} \def\cprime{$'$}
  \def\cprime{$'$} \def\cprime{$'$}
\begin{thebibliography}{10}

\bibitem{ades2013exploration}
{\sc M~Ades and PJ~Van~Leeuwen}, {\em An exploration of the equivalent weights
  particle filter}, Quarterly Journal of the Royal Meteorological Society, 139
  (2013), pp.~820--840.

\bibitem{ApteJV08}
{\sc A.~Apte, C.K.R.T Jones, A.M. Stuart, and J.~Voss}, {\em Data assimilation:
  mathematical and statistical perspectives}, Int. J. Num. Meth. Fluids, 56
  (2008), pp.~1033--1046.

\bibitem{augustin2008polynomial}
{\sc Florian Augustin, A~Gilg, M~Paffrath, P~Rentrop, and U~Wever}, {\em
  Polynomial chaos for the approximation of uncertainties: chances and limits},
  European Journal of Applied Mathematics, 19 (2008), pp.~149--190.

\bibitem{BC09}
{\sc A.~Bain and D.~Crisan}, {\em Fundamentals of Stochastic Filtering},
  Springer, 2009.

\bibitem{bao2014hybrid}
{\sc Feng Bao, Yanzhao Cao, Clayton Webster, and Guannan Zhang}, {\em A hybrid
  sparse-grid approach for nonlinear filtering problems based on
  adaptive-domain of the {Z}akai equation approximations}, SIAM/ASA Journal on
  Uncertainty Quantification, 2 (2014), pp.~784--804.

\bibitem{BLB08}
{\sc P.~Bickel, B.~Li, and T.~Bengtsson}, {\em Sharp failure rates for the
  bootstrap particle filter in high dimensions}, IMS Collections: Pushing the
  Limits of Contemporary Statistics, 3 (2008), pp.~318--329.

\bibitem{blomker2012accuracy}
{\sc D~Bl{\"o}mker, K~Law, AM~Stuart, and KC~Zygalakis}, {\em Accuracy and
  stability of the continuous-time {3DVAR} filter for the {Navier-Stokes}
  equation}, Nonlinearity, 26 (2012), p.~2193.

\bibitem{bobrowski2005functional}
{\sc Adam Bobrowski}, {\em Functional analysis for probability and stochastic
  processes: an introduction}, Cambridge University Press, 2005.

\bibitem{bou2010pathwise}
{\sc Nawaf Bou-Rabee and Eric Vanden-Eijnden}, {\em Pathwise accuracy and
  ergodicity of metropolized integrators for {SDE}s}, Communications on Pure
  and Applied Mathematics, 63 (2010), pp.~655--696.

\bibitem{branicki2012fundamental}
{\sc Michal Branicki and Andrew~J Majda}, {\em Fundamental limitations of
  polynomial chaos for uncertainty quantification in systems with intermittent
  instabilities}, Comm. Math. Sci, 11 (2012).

\bibitem{brett2012accuracy}
{\sc CEA Brett, KF~Lam, KJH Law, DS~McCormick, MR~Scott, and AM~Stuart}, {\em
  Accuracy and stability of filters for dissipative {PDE}s}, Physica D:
  Nonlinear Phenomena, 245 (2012), pp.~34--45.

\bibitem{bungartz2004sparse}
{\sc Hans-Joachim Bungartz and Michael Griebel}, {\em Sparse grids}, Acta
  Numerica, 13 (2004), pp.~147--269.

\bibitem{burgers1998analysis}
{\sc Gerrit Burgers, Peter Jan~van Leeuwen, and Geir Evensen}, {\em Analysis
  scheme in the ensemble {K}alman filter}, Monthly Weather Review, 126 (1998),
  pp.~1719--1724.

\bibitem{cappe2005inference}
{\sc Olivier Capp{\'e}, Eric Moulines, and Tobias Ryd{\'e}n}, {\em Inference in
  hidden {M}arkov models}, Springer, 2005.

\bibitem{CK04}
{\sc A.J. Chorin and P.~Krause}, {\em Dimensional reduction for a {B}ayesian
  filter}, Proc. Nat. Acad. Sci., 101 (2004), pp.~15013--15017.

\bibitem{chorin2010implicit}
{\sc Alexandre Chorin, Matthias Morzfeld, and Xuemin Tu}, {\em Implicit
  particle filters for data assimilation}, Communications in Applied
  Mathematics and Computational Science, 5 (2010), pp.~221--240.

\bibitem{del2001stability}
{\sc Pierre Del~Moral and Alice Guionnet}, {\em On the stability of interacting
  processes with applications to filtering and genetic algorithms}, in Annales
  de l'Institut Henri Poincar{\'e} (B) Probability and Statistics, vol.~37,
  Elsevier, 2001, pp.~155--194.

\bibitem{dobrushin1956central}
{\sc Roland~L Dobrushin}, {\em Central limit theorem for nonstationary {M}arkov
  chains. i,ii}, Theory of Probability \& Its Applications, 1 (1956),
  pp.~65--80.

\bibitem{doucet2000sequential}
{\sc Arnaud Doucet, Simon Godsill, and Christophe Andrieu}, {\em On sequential
  {Monte Carlo sampling methods for Bayesian filtering}}, Statistics and
  computing, 10 (2000), pp.~197--208.

\bibitem{DdFG01}
{\sc N.~Doucet, A. de~Frietas and N.~Gordon}, {\em Sequential Monte Carlo in
  Practice}, Springer-Verlag, 2001.

\bibitem{el2012bayesian}
{\sc Tarek~A El~Moselhy and Youssef~M Marzouk}, {\em Bayesian inference with
  optimal maps}, Journal of Computational Physics, 231 (2012), pp.~7815--7850.

\bibitem{el2012bayesian2}
\leavevmode\vrule height 2pt depth -1.6pt width 23pt, {\em Bayesian filtering
  with optimal maps}, In preparation,  (2014).

\bibitem{ernst2014bayesian}
{\sc Oliver~G Ernst, Bj{\"o}rn Sprungk, and Hans-J{\"o}rg Starkloff}, {\em
  Bayesian inverse problems and {K}alman filters}, in Extraction of
  Quantifiable Information from Complex Systems, Springer, 2014, pp.~133--159.

\bibitem{evans}
{\sc L.C. Evans}, {\em Partial Differential Equations}, AMS, Providence, Rhode
  Island, 1998.

\bibitem{evans2012introduction}
{\sc Lawrence~C Evans}, {\em An introduction to stochastic differential
  equations}, vol.~82, American Mathematical Soc., 2012.

\bibitem{evensen1994sequential}
{\sc Geir Evensen}, {\em Sequential data assimilation with a nonlinear
  quasi-geostrophic model using monte carlo methods to forecast error
  statistics}, Journal of Geophysical Research: Oceans (1978--2012), 99 (1994),
  pp.~10143--10162.

\bibitem{eyink2004mean}
{\sc Gregory~L Eyink, Juan~M Restrepo, and Francis~J Alexander}, {\em A mean
  field approximation in data assimilation for nonlinear dynamics}, Physica D:
  Nonlinear Phenomena, 195 (2004), pp.~347--368.

\bibitem{hayden2011discrete}
{\sc K.~Hayden, E.~Olson, and E.S. Titi}, {\em Discrete data assimilation in
  the {L}orenz and 2d {N}avier-{S}tokes equations}, Physica D: Nonlinear
  Phenomena,  (2011).

\bibitem{hoteit2011particle}
{\sc Ibrahim Hoteit, Xiaodong Luo, and Dinh-Tuan Pham}, {\em {Particle Kalman
  filtering: A nonlinear Bayesian framework for ensemble Kalman filters}},
  arXiv:1108.0168,  (2011).

\bibitem{jaz70}
{\sc A.H. Jazwinski}, {\em Stochastic processes and filtering theory}, vol.~63,
  Academic Pr, 1970.

\bibitem{kaipio2005statistical}
{\sc J.P. Kaipio and E.~Somersalo}, {\em Statistical and computational inverse
  problems}, Springer Science+ Business Media, Inc., 2005.

\bibitem{kalman1960new}
{\sc Rudolph~Emil Kalman et~al.}, {\em A new approach to linear filtering and
  prediction problems}, Journal of basic Engineering, 82 (1960), pp.~35--45.

\bibitem{kal03}
{\sc E.~Kalnay}, {\em Atmospheric Modeling, Data Assimilation and
  Predictability}, Cambridge, 2003.

\bibitem{KLS13}
{\sc D.T.B. Kelly, K.J.H. Law, and A.M. Stuart}, {\em Well-posedness and
  accuracy of the ensemble {Kalman} filter in discrete and continuous time}.
\newblock Submitted.

\bibitem{kwiatkowski2015convergence}
{\sc Evan Kwiatkowski and Jan Mandel}, {\em Convergence of the square root
  ensemble {K}alman filter in the large ensemble limit}, SIAM/ASA Journal on
  Uncertainty Quantification, 3 (2015), pp.~1--17.

\bibitem{laplace1986memoir}
{\sc Pierre~Simon Laplace}, {\em Memoir on the probability of the causes of
  events}, Statistical Science,  (1986), pp.~364--378.

\bibitem{law2012evaluating}
{\sc KJH Law and AM~Stuart}, {\em Evaluating data assimilation algorithms},
  Monthly Weather Review, 140 (2012), pp.~3757--3782.

\bibitem{lsz2015}
{\sc Kody J~H Law, Andrew~M Stuart, and Konstantinos~C Zygalakis}, {\em Data
  Assimilation: A Mathematical Introduction}, Springer Texts in Applied
  Mathematics, 2015.

\bibitem{le2011large}
{\sc Fran{\c{c}}ois Le~Gland, Val{\'e}rie Monbet, Vu-Duc Tran, et~al.}, {\em
  Large sample asymptotics for the ensemble {K}alman filter}, The Oxford
  Handbook of Nonlinear Filtering,  (2011), pp.~598--631.

\bibitem{li2009generalized}
{\sc Jia Li and Dongbin Xiu}, {\em A generalized polynomial chaos based
  ensemble {K}alman filter with high accuracy}, Journal of computational
  physics, 228 (2009), pp.~5454--5469.

\bibitem{luenberger1968optimization}
{\sc David~G Luenberger}, {\em Optimization by vector space methods}, John
  Wiley \& Sons, 1969.

\bibitem{mandel2009ensemble}
{\sc Jan Mandel and Jonathan~D Beezley}, {\em An ensemble {Kalman-particle
  predictor-corrector filter for non-Gaussian data assimilation}}, in
  Computational Science--ICCS 2009, Springer, 2009, pp.~470--478.

\bibitem{mandel2011convergence}
{\sc Jan Mandel, Loren Cobb, and Jonathan~D Beezley}, {\em On the convergence
  of the ensemble {K}alman filter}, Applications of Mathematics, 56 (2011),
  pp.~533--541.

\bibitem{Marko_Vill00}
{\sc P.~A. Markowich and C.~Villani}, {\em On the trend to equilibrium for the
  {F}okker-{P}lanck equation: an interplay between physics and functional
  analysis}, Mat. Contemp., 19 (2000), pp.~1--29.

\bibitem{mattingly2002ergodicity}
{\sc Jonathan~C Mattingly, Andrew~M Stuart, and Desmond~J Higham}, {\em
  Ergodicity for sdes and approximations: locally lipschitz vector fields and
  degenerate noise}, Stochastic processes and their applications, 101 (2002),
  pp.~185--232.

\bibitem{Oks98}
{\sc B.~Oksendal}, {\em Stochastic differential equations}, Universitext,
  Springer, sixth~ed., 2003.
\newblock An introduction with applications.

\bibitem{pajonk2012stochastic}
{\sc Oliver Pajonk}, {\em Stochastic Spectral Methods for Linear {B}ayesian
  Inference}, PhD thesis.

\bibitem{pajonk2012deterministic}
{\sc Oliver Pajonk, Bojana~V Rosi{\'c}, Alexander Litvinenko, and Hermann~G
  Matthies}, {\em A deterministic filter for non-{Gaussian B}ayesian
  estimationÑ applications to dynamical system estimation with noisy
  measurements}, Physica D: Nonlinear Phenomena, 241 (2012), pp.~775--788.

\bibitem{pajonk2013sampling}
{\sc Oliver Pajonk, Bojana~V Rosi{\'c}, and Hermann~G Matthies}, {\em
  Sampling-free linear {B}ayesian updating of model state and parameters using
  a square root approach}, Computers \& Geosciences, 55 (2013), pp.~70--83.

\bibitem{rebeschini2013can}
{\sc Patrick Rebeschini and Ramon van Handel}, {\em Can local particle filters
  beat the curse of dimensionality?}, arXiv:1301.6585,  (2013).

\bibitem{reich2012gaussian}
{\sc Sebastian Reich}, {\em A {G}aussian-mixture ensemble transform filter},
  Quarterly Journal of the Royal Meteorological Society, 138 (2012),
  pp.~222--233.

\bibitem{Ris84}
{\sc H.~Risken}, {\em The {F}okker-{P}lanck equation}, vol.~18 of Springer
  Series in Synergetics, Springer-Verlag, Berlin, 1989.

\bibitem{salman2008hybrid}
{\sc H~Salman}, {\em A hybrid grid/particle filter for {L}agrangian data
  assimilation. i: Formulating the passive scalar approximation}, Quarterly
  Journal of the Royal Meteorological Society, 134 (2008), pp.~1539--1550.

\bibitem{santitissadeekorn2014two}
{\sc Naratip Santitissadeekorn and Chris Jones}, {\em Two-state filtering for
  joint state-parameter estimation}, arXiv:1403.5989,  (2014).

\bibitem{sapsis2013blended}
{\sc Themistoklis~P Sapsis and Andrew~J Majda}, {\em Blended reduced subspace
  algorithms for uncertainty quantification of quadratic systems with a stable
  mean state}, Physica D: Nonlinear Phenomena,  (2013).

\bibitem{solonen2012variational}
{\sc Antti Solonen, Heikki Haario, Janne Hakkarainen, Harri Auvinen, Idrissa
  Amour, and Tuomo Kauranne}, {\em Variational ensemble {Kalman filtering using
  limited memory BFGS}}, Electronic Transactions on Numerical Analysis, 39
  (2012), pp.~271--285.

\bibitem{Stuart10}
{\sc A.M. Stuart}, {\em Inverse problems: a {B}ayesian approach}, Acta
  Numerica, 19 (2010), pp.~451--559.

\bibitem{uboldi2005developing}
{\sc F~Uboldi, A~Trevisan, A~Carrassi, et~al.}, {\em Developing a dynamically
  based assimilation method for targeted and standard observations}, Nonlinear
  Processes in Geophysics, 12 (2005), pp.~149--156.

\bibitem{vanden2012data}
{\sc Eric Vanden-Eijnden and Jonathan Weare}, {\em Data assimilation in the low
  noise, accurate observation regime with application to the {K}uroshio
  current}, Monthly Weather Review, 141 (2012), p.~1.

\end{thebibliography}

\end{document}